\documentclass[reqno,12pt]{amsart}

\usepackage{color} % black,white,red,green,blue,cyan,magenta,yellow
\usepackage{ifpdf}
\ifpdf
    \usepackage[pdftex]{graphicx}
    \usepackage[pdftex]{hyperref}
    \hypersetup{
        unicode=false,          % non-Latin characters in Acrobat’s bookmarks
        pdftoolbar=true,        % show Acrobat’s toolbar?
        pdfmenubar=true,        % show Acrobat’s menu?
        pdffitwindow=false,     % window fit to page when opened
        pdfstartview={FitH},    % fits the width of the page to the window
        pdftitle={MCP Article},      % title
        pdfauthor={Michael Holst},   % author
        pdfsubject={Mathematics},    % subject of the document
        pdfcreator={Michael Holst},  % creator of the document
        pdfproducer={Michael Holst}, % producer of the document
        pdfkeywords={PDE, analysis, mathematical physics}, % list of keywords
        pdfnewwindow=true,      % links in new window
        colorlinks=true,        % false: boxed links; true: colored links
        linkcolor=red,          % color of internal links
        citecolor=blue,         % color of links to bibliography
        filecolor=magenta,      % color of file links
        urlcolor=cyan           % color of external links
    }

\else
    \usepackage{graphicx}
    \usepackage{pstricks}
    
    \newcommand{\href}[2]{#2}
\fi

\usepackage{times}
\usepackage{amsfonts}

\usepackage{amsmath}
\usepackage{amsthm}
\usepackage{amssymb}
\usepackage{amsbsy}
\usepackage{amscd}

\usepackage{enumerate}
\usepackage{verbatim}
\usepackage{subfigure}

\newtheorem{theorem}{Theorem}[section]

\newtheorem{lemma}[theorem]{Lemma}

\newtheorem{proposition}[theorem]{Proposition}

\newtheorem{definition}[theorem]{Definition}

\newtheorem{remark}[theorem]{Remark}

\numberwithin{equation}{section}  %amsmath command: tie counter to section

  \newcounter{mnote}
  \newcounter{mnoteR}
\let\oldmarginpar\marginpar
  \renewcommand\marginpar[1]{\-\oldmarginpar[\raggedleft\footnotesize #1]%
  {\raggedright\footnotesize #1}}

\definecolor{myblue}{rgb}{0.2,0.2,0.7}
\definecolor{mygreen}{rgb}{0,0.6,0}
\definecolor{mycyan}{rgb}{0,0.6,0.6}
\definecolor{myred}{rgb}{0.9,0.2,0.2}
\definecolor{mymagenta}{rgb}{0.9,0.2,0.9}
\definecolor{mywhite}{rgb}{1.0,1.0,1.0}
\definecolor{myblack}{rgb}{0.0,0.0,0.0}

\newcommand{\beq}{\begin{equation}}
\newcommand{\eeq}{\end{equation}}
\newcommand{\beqa}{\begin{eqnarray}}
\newcommand{\eeqa}{\end{eqnarray}}
\newcommand{\cD}{{\mathcal D}}
\newcommand{\cE}{{\mathcal E}}
\newcommand{\cF}{{\mathcal F}}
\newcommand{\cG}{{\mathcal G}}

\newcommand{\cL}{{\mathcal L}}
\newcommand{\cM}{{\mathcal M}}
\newcommand{\cN}{{\mathcal N}}
\newcommand{\cO}{{\mathcal O}}

\newcommand{\cS}{{\mathcal S}}

\newcommand{\cY}{{\mathcal Y}}

\newcommand{\bW}{{\bf W}}

\newcommand{\bj}{{\bf j}}

\newcommand{\bw}{{\bf w}}

\def\Om{\Omega}

\newcommand{\e}{\epsilon}
\newcommand{\ol}[1]{\overline{#1}}
\begin{document}

\title[Generalized Solutions to Semilinear Elliptic PDE]
      {Generalized Solutions to Semilinear Elliptic PDE \\
       with Applications to the Lichnerowicz Equation}

\author[M. Holst]{Michael Holst}
\email{mholst@math.ucsd.edu}

\author[C. Meier]{Caleb Meier}
\email{meiercaleb@gmail.com}

\address{Department of Mathematics\\
         University of California San Diego\\ 
         La Jolla CA 92093}

\thanks{MH was supported in part by NSF Awards~1217175 and 1065972.}
\thanks{CM was supported in part by NSF Award~1065972.}

\date{\today}


\begin{abstract}
In this article we investigate the existence of a solution to a semilinear, 
elliptic, partial differential equation with distributional coefficients 
and data.
The problem we consider is a generalization of the Lichnerowicz equation
that one encounters in studying the constraint equations in general relativity.
Our method for solving this problem consists of solving a net of regularized,
semilinear problems with data obtained by smoothing the original, 
distributional coefficients.
In order to solve these regularized problems, we develop {\em a priori}
$L^{\infty}$-bounds and sub- and super-solutions and then apply 
a fixed-point argument for order-preserving maps.
We then show that the net of solutions obtained through this 
process satisfies certain decay estimates by determining estimates
for the sub- and super-solutions and by utilizing classical, {\em a priori}
elliptic estimates.
The estimates for this net of solutions allow us to regard this collection 
of functions as a solution in a Colombeau-type algebra.
We motivate this Colombeau algebra framework by first solving an 
ill-posed critical exponent problem.
To solve this ill-posed problem, we use a collection of smooth, 
``approximating'' problems and then use the resulting sequence
of solutions and a compactness argument to obtain a solution to the 
original problem.
This approach is modeled after the more general Colombeau framework that 
we develop, and it conveys the potential that solutions in these abstract 
spaces have for obtaining classical solutions to ill-posed nonlinear 
problems with irregular data.   
\end{abstract}

\maketitle

%\clearpage

\vspace*{-1.2cm}
{\tiny
\tableofcontents
}
%\vspace*{-0.5cm}

%%%%%%%%%%%%%%%%%%%%%%%%%%%%%%%%%%%%%%%%%%%%%%%%%%%%%%%%%%%%%%%%%%%%%%%%%%%
\section{Introduction}

%\mnote{Changed introduction.  Added some background on the constraints and rough solution theory for motivation.}
%\mnote{OK; Some small changes made, added cite to Yvonne's 2004 rough paper, etc. -mjh}
The goal of this paper is to develop a framework to extend the rough solution theory for the conformally rescaled
Einstein constraint equations when the mean curvature is constant.  In the event that the mean curvature is constant,
the conformally rescaled constraint equations decouple, leaving only a semilinear elliptic equation to solve.  In attempting
to extend the rough solution theory in this case, one is confronted with the problem of solving a semilinear elliptic equation 
with distributional coefficients.  If these coefficients do not lie in certain Sobolev spaces with somewhat exacting restrictions on their
indices, the resulting elliptic problem will not be well-defined in the normal weak sense.  In an effort to circumvent these
restrictions on the Sobolev classes of our coefficients, we develop a method to reformulate the
ill-posed, semilinear PDE with singular coefficients as a PDE in what is known as a Colombeau algebra.
These Colombeau algebras contain the space of distributions via an embedding, so one solves the PDE in the Colombeau algebra
and attempts to associate the generalized Colombeau solution with 
a distribution, thereby obtaining a distributional solution to the original ill-defined problem.

\subsection{ The Einstein Constraint Equations and Conformal Method}
The Einstein field equation $G_{\mu\nu} = \kappa T_{\mu\nu}$ can be formulated as an initial value (or Cauchy)
problem where the initial data consists of a Riemannian metric $\hat{g}_{ab}$ and a symmetric tensor $\hat{k}_{ab}$ on a specified $3$-dimensional
manifold $\cM$ \cite{HE75,RW84}.  However, one is not
able to freely specify such initial data.  Like Maxwell's equations, the initial data $\hat{g}_{ab}$ and $\hat{k}_{ab}$ must satisfy 
constraint equations, where the constraints take the form
\begin{align}
\hat{R}+\hat{k}^{ab}\hat{k}_{ab}+\hat{k}^2 = 2\kappa\hat{\rho}, \label{eq1:5aug12}\\
\hat{D}_b\hat{k}^{ab}-\hat{D}^a\hat{k} = \kappa \hat{j}^a. \label{eq2:5aug12}
\end{align}
Here $\hat{R}$ and $\hat{D}$ are the scalar curvature and covariant derivative associated with $\hat{g}_{ab}$, $\hat{k}$ is the trace of
$\hat{k}_{ab}$ and $\hat{\rho}$ and $\hat{j}^a$ are matter terms obtained by contracting $T_{\mu\nu}$ with a vector field normal
to $\cM$.  As the
Cauchy formulation of the Einstein field equations is one of the most important means of modeling and studying astrophysical 
phenomena, knowledge of the constraint equations is very important because of the influence that solutions to these equations has on
solutions to the evolution problem.
Moreover, a number of central questions in general relativity are addressed entirely through the study of the constraint equations alone (cf.~\cite{BI04} for discussion).

Equation \eqref{eq1:5aug12} is known as the Hamiltonian constraint while \eqref{eq2:5aug12} is known as the momentum
constraint, and collectively the two expressions are known as the Einstein constraint equations.
These equations form an underdetermined system of four equations to be solved for twelve unknowns $\hat{g}_{ab}$ and $\hat{k}_{ab}$.  
In order to transform the constraint equations into a determined system, one divides the unknowns into freely specifiable data and determined data by using what is known as the conformal
method.  In this method introduced by
Lichnerowicz \cite{AL44} and York \cite{JY71}, we assume that the metric $\hat{g}_{ab}$ is known up to a conformal factor and that
the trace $\hat{k}$ and a term proportional to a trace-free divergence-free part of $\hat{k}_{ab}$ is known.
Therefore the determined data in this formulation of the constraints is the conformal factor $\phi$ and a vector field $\bw$ whose symmetrized 
derivative represents the undetermined portion of $\hat{k}_{ab}$.  One obtains the following system   
\begin{align}
&-8\Delta\phi + R\phi + \frac23\tau^2\phi^5-[\sigma_{ab} +(\cL \bw)_{ab}][\sigma^{ab} +(\cL \bw)^{ab}]\phi^{-7}-2\kappa\rho\phi^{-3} = 0\label{eq4:7aug12}, \\
&-D_b(\cL \bw)^{ab}+\frac23D^a\tau \phi^6+\kappa j^a = 0, \label{eq5:7aug12}
\end{align}
which forms a determined, coupled nonlinear system of elliptic 
equations that is referred to as the conformal, transverse, traceless (CTT) formulation of the constraints.

In equations \eqref{eq4:7aug12}-\eqref{eq5:7aug12} the quantities $g_{ab}, \sigma_{ab}$, $\tau$, $\rho$, $j^a$ are freely specified and satisfy
\begin{align}\label{eq1:15nov12}
&\hat{g}_{ab} = \phi^4g_{ab}, \quad \hat{k}^{ab} = \phi^{-10}[\sigma^{ab}+(\cL w)^{ab}] + \frac13\phi^{-4}\tau g^{ab},\\
&\hat{j}^a = \phi^{-10}j^a,\quad \hspace{14mm} \quad \hat{\rho} = \phi^{-8}\rho,
\end{align} 
and $\Delta, \cL, D$ and $R$ are the Laplace-Beltrami operator, conformal Killing operator, covariant derivative and scalar curvature associated with $g_{ab}$.  
For a given choice of 
$g_{ab}, \sigma_{ab}, \\
\tau, \rho, j^a$, if one can solve \eqref{eq4:7aug12}-\eqref{eq5:7aug12} for $\phi$
and $\bw$, they obtain a solution to the constraint equations \eqref{eq1:5aug12}-\eqref{eq2:5aug12} by using Eq. \eqref{eq1:15nov12} to reconstruct the
physical solutions $\hat{g}_{ab}$ and $\hat{k}_{ab}$.  

\subsection{Solution Theory for the CTT Formulation}

The solution theory for the CTT formulation of the Einstein constraint equations on a closed manifold $\cM$
can be roughly classified according to the Yamabe class of the given metric $g_{ab}$, the properties of $\tau$ (the mean extrinsic curvature) and
the regularity of the specified data $(g_{ab},\tau, \sigma,\rho,\bj)$.  The mean curvature plays perhaps the largest role.  
If the mean curvature is constant, then the analysis of the conformal formulation simplifies greatly
because the Hamiltonian constraint and the momentum constraint decouple, leaving a single semilinear
elliptic PDE to analyze.  
For $C^2$ metrics, the classical solution theory for the conformal formulation with constant mean curvature (CMC) is now understood for all three Yamabe classes, and is summarized in \cite{JI95}.
The solution theory for low-regularity data $(g_{ab},\tau, \sigma, \rho,\bj)$, or so-called ``rough solution theory", is also well developed in the CMC case.
The most complete rough solution theory to date appears in~\cite{HNT09}, and allows for metrics $g_{ab} \in W^{s,p}$, with any pair $s,p$ satisfying $s> \frac3p$ and specified data
$\sigma,~ \rho,~ \bj$ satisfying
\begin{align}\label{conditions1}
&\bullet \sigma \in W^{e-1,q},\hspace{3.75 in}\\
&\bullet \rho \in W^{s-2,p}, \nonumber\\
&\bullet \bj \in \bW^{e-2,q},\nonumber
\end{align}
where $q$ and $e$ satisfy
\begin{align}\label{conditions2} 
&\bullet \frac1q \in (0,1)\cap [\frac{3-p}{3p},\frac{3+p}{3p}]\cap [\frac{1-d}{3},\frac{3+sp}{6p}),\hspace{1.50 in}\\
&\bullet e \in [1,\infty)\cap [s-1,s]\cap [\frac3q+d-1,\frac3q+d]\cap(\frac3q+\frac d2,\infty),\nonumber
\end{align}
with $d = s-p$.  There are also some additional assumptions on the Yamabe class of
$g_{ab}$ and the sign of $\sigma,~ \rho ~\text{and}~ \bj$. 
(cf. \cite{CB04,DMa05,DMa06,HNT08,HNT09}).

\subsection{Rough Solutions to the Constraint Equations}   

There is an incentive to develop a low regularity solution framework for the Einstein 
field equations to model plausible astronomical
phenomena such as cosmic strings and gravitational waves \cite{GKOS01}.
The solutions to the constraint equations not only place a restriction on which metrics and extrinsic 
curvature tensors can be considered as initial data, but they also determine the function spaces of maximally
globally hyperbolic solutions to the evolution problem \cite{BI04}.  
The solution theory for the constraint equations must therefore keep pace with the theory for the evolution equations, in order to avoid limiting the further theoretical development of the theory for the evolution problem.
Historically, the rough solution theory of the constraints has in fact lagged behind that of the evolution problem. 
The local well-posedness result for quasilinear hyperbolic systems in \cite{HKM76} allows for initial data $(g,K)$ in $H^{s}\times H^{s-1}$ for
$s>\frac52$; however it was not until \cite{CB04,DMa05,DMa06} that solutions of this regularity existed to the constraint
equations, and even these initial results were restricted to CMC solutions.  
Low regularity solutions became increasingly important when Klainerman and Rodnianski 
developed {\em a priori} estimates in \cite{KR05} for the time existence of 
solutions to the vacuum Einstein equations in terms
of the $H^{s-1}\times H^{s-1}$ norm of $( Dg, K)$, again with $s>2$.
This prompted Maxwell's work on the CMC case in \cite{DMa05} and \cite{DMa06},
Choquet-Bruhat's work on the CMC case in \cite{CB04}, 
and Holst's et al.'s work on both the CMC and non-CMC 
cases in \cite{HNT08,HNT09}.

One of the difficulties associated with obtaining rough solutions to the conformal formulation is that in general, Sobolev spaces are not closed under multiplication.
With the exception of the Banach spaces $W^{s,p}(\cM)$ with $s> d/p$ (where $d$ is the spatial dimension), the product of two Sobolev functions in a given space will not in general lie in that space.
%\mnoteR{Changed to full Banach space case, since that is limitation in most general rough theory in \cite{HNT09}.  -mjh}
This restriction is a by-product of a more general problem, which is that in general, there is no well-behaved definition of distributional multiplication that allows for the multiplication of arbitrary distributions.
One is instead confined to work with subspaces such as Sobolev spaces where point-wise multiplication is only well-defined for certain choices of Sobolev indices.
This greatly limits the Sobolev spaces that one considers 
when attempting to develop a weak formulation of a given elliptic partial differential
equation, and in particular, places a restriction on the regularity of the specified data $(g_{ab},\tau,\sigma,\rho,\bj)$
of the CTT equations.  

In order to overcome these limitations, we developed a framework to solve semilinear elliptic
problems similar to the Hamiltonian constraint in generalized function spaces known as Colombeau algebras.  
This work is a natural extension of the work done by Mitrovic and Pilipovic in \cite{MP06}, where the authors found generalized
solutions to linear, elliptic equations with distributional coefficients.  The advantage of solving PDE in these generalized
function spaces is that it allows one to circumvent the restrictions associated with Sobolev coefficients and
data, and thereby consider problems with coefficients and data of much lower regularity.  

\subsection{Low Regularity Semilinear Elliptic Problems}

If the mean curvature $\tau$ is constant, the CTT formulation \eqref{eq4:7aug12}-\eqref{eq5:7aug12}
reduces to 
\begin{align}
-\Delta \phi + \frac18 R\phi + \frac{1}{12}\tau^2\phi^5-\sigma^2\phi^{-7} - 2\kappa\rho\phi^{-3} = 0.
\end{align}
Locally, on a given chart element $\Om = \psi(U)$, this problem assumes the form
\begin{align}\label{eq1:18dec11}
-\sum_{i,j=1}^3 D_i(a^{ij} D_j u) &+ b_1u^5-b_2u^{-7}-b_3u^{-3} = 0 ~~~\text{on $\Om$},\\
 &u = \varphi \quad \text{on $\partial\Om$} \nonumber.
\end{align}
In an effort to extend the rough solution theory of the constraints, 
we are interested in solving \eqref{eq1:18dec11} with minimal regularity assumptions on the 
coefficients $a^{ij}, b_1, b_2 , b_3$ and boundary data $\varphi$.  Therefore,
in this paper we consider a family of elliptic, semilinear Dirichlet problems
that are of the form
\begin{align}
\label{problem}
-\sum_{i,j=1}^ND_i(a^{ij}D_ju) &+ \sum^K_{i=1}b^i u^{n_i} = 0 \quad \text{ in $\Om$},\\
&u= \rho \quad \text{on $\partial\Om$}, \nonumber
\end{align}
where $a^{ij}, b^i$ and $\rho$ are potentially distributional and 
$n_i \in \mathbb{Z}$ for each $i$.  

The main contributions of this article are an existence result for \eqref{problem} in a
Colombeau-type algebra, and an existence result in $H^{1}(\Om)$ for an ill-posed, critical exponent
problem of the form
\begin{align}\label{eq3:25oct11}
-\Delta  u +au^m&+bu^i =0 ~~~~\text{in $\Om$},\\
            &u = \rho \quad \text{on $\partial\Om$}, \nonumber
\end{align}
where $m \ge 5$, $1\le i \le 4$ are in $\mathbb{N}$, $\Om \subset \mathbb{R}^3$, $b\in L^{\infty}(\Om)$ and $a \in L^p(\Om)$ with $\frac65 \le p < \infty$.
The framework we use to prove existence for \eqref{problem} consists 
of embedding the singular data and coefficients into a Colombeau-type algebra 
so that multiplication of the distributional coefficients is well-defined.
To solve \eqref{eq3:25oct11}, we do not explicitly require the Colombeau machinery
that we develop to solve \eqref{problem}, 
but we use similar ideas to produce a sequence of functions
that converge to a solution of \eqref{eq3:25oct11} in $H^{1}(\Om)$. 

The Colombeau solution framework for this paper is based mainly on the ideas found in~\cite{MP06}.
Here we extend the work done by Mitrovic and Pilipovic 
in~\cite{MP06} to include a certain collection of semilinear problems.
While Pilipovic and Scarpalezos solved a divergent type, quasilinear
problem in a Colombeau type algebra in \cite{PS06}, the class of nonlinear problems we consider
here does not fit naturally into that framework.  Here we provide a solution
method that is distinct from those posed in \cite{PS06} and \cite{MP06} that is better suited for the class of semilinear
problems that we are interested in solving.
The set up of our problem is completely similar to the 
set-up in~\cite{MP06}:
given the semilinear Dirichlet problem in \eqref{problem}, we consider the family of problems 
\begin{align}
\label{problem2}
P_{\e}(x,D)u_{\e} & = f_{\e}(x,u_{\e}) \quad \text{on $\Om$},\\ 
u_{\e}  &= \rho_{\e} \quad \text{on $\partial\Om$},\nonumber
\end{align}
where $f_{\e}, h_{\e}$, and $P_{\e}(x,D)$ are obtained by convolving the data and coefficients of 
\eqref{problem} with a certain mollifier.  Thus a solution to the problem in a Colombeau algebra is a net of solutions to the 
above family satisfying certain growth estimates in $\e$.  This is discussed in detail in Sections \ref{Colombeau} and \ref{netsofproblems}. 
This basic concept underlies both the solution process in our paper and in \cite{MP06} and \cite{PS06}.
However it is our solution process
in the Colombeau algebra that is quite distinct from that laid out 
in~\cite{MP06}, where the authors used linear elliptic theory to determine
a family of solutions and then classical elliptic, {\em a priori} estimates to prove certain growth estimates.  Most notably, the authors developed
a precise maximum principle-type argument necessary to obtain polynomial growth estimates required to find a solution.  Our strategy for solving~\eqref{problem}
differs in a number of ways.  First, in Section \ref{bounds} we develop a family of { \em a priori} $L^{\infty}$ bounds to the family of problems \eqref{problem2}.
Then in Section \ref{supersolution} we show that these estimates determine sub- and super-solutions to \eqref{problem2}.
We then employ the method of sub- and super-solutions in Section \ref{subsolution} to determine a family of solutions.  
Finally, $\e$-growth estimates on the sub- and super-solutions are established in Section \ref{supersolution}, 
and in Section \ref{results} these estimates are used in conjunction with the {\em a priori}
estimates in Section \ref{holder} to prove the necessary $\e$-growth estimates on our family of solutions.

This paper can be broken down into two distinct, but related parts.  The first part is dedicated to solving 
\eqref{eq3:25oct11}. 
Our solution to this problem does not explicitly require the techniques that we develop to solve problems with distributional 
data in Colombeau algebras and only relies on standard elliptic PDE theory.  However, the ideas that we use to solve the problem are closely related:    
we obtain a solution by solving a family of problems similar to \eqref{problem2} and then show that these solutions
converge to a function in $H^{1}(\Om)$.  Therefore, we present our existence result for \eqref{eq3:25oct11} 
first to convey the benefit that the more general Colombeau solution strategy has, not only for solving problems
in the Colombeau Algebra, but also for obtaining solutions in more classical spaces.  The remainder
of the paper is dedicated to developing the Colombeau framework described in the preceding paragraph.  This
consists of defining an algebra appropriate for a Dirichlet problem and properly defining a semilinear elliptic
problem in the algebra.  Once a well-posed elliptic problem in the Colombeau algebra has been formed, 
we discuss the conditions under which the problem has a solution in the algebra
and finally, describe how to translate a given problem of the form \eqref{problem}
into a problem that can be solved in the algebra.  It should be noted that while the intention is
to find solutions to \eqref{problem}, the main result pertaining to Colombeau algebras
in this paper is Theorem~\ref{thm1june27}, which is the main solution result
for semilinear problems in our particular Colombeau algebra. 

{\em Outline of the paper.}
The remainder of the paper is structured as follows.
In Section~\ref{Example1} we motivate this article by proving the 
existence of a solution to \eqref{eq3:25oct11}.
In Section~\ref{prelim} we state a number of preliminary results and develop 
the technical tools required to solve~\eqref{problem}.
Among these tools and results are the explicit {\em a priori} estimates 
found in~\cite{MP06} and a description of the Colombeau framework in which the 
coefficients and data will be embedded.
Then in Section~\ref{overview} we state the main existence result in 
Theorem~\ref{thm1june27}, give a statement and proof of the method of 
sub- and super solutions in Theorem~\ref{thm2june27}, and then give an 
outline of the method of proof of Theorem~\ref{thm1june27}.
Following our discussion of elliptic problems in Colombeau algebras, we 
discuss a method to embed \eqref{problem} into the algebra to apply our 
Colombeau existence theory.
The remainder of the paper is dedicated to developing the tools to prove 
Theorem~\ref{thm1june27}.
In Section~\ref{bounds1} we determine {\em a priori} $L^{\infty}$ bounds of 
solutions to our semilinear problem and a net of sub- and super-solutions 
satisfying explicit $\e$-growth estimates.
Finally, in Section~\ref{results} we utilize the results from 
Section~\ref{bounds1} to prove the main result outlined in 
Section~\ref{overview}.

%%%%%%%%%%%%%%%%%%%%%%%%%%%%%%%%%%%%%%%%%%%%%%%%%%%%%%%%%%%%%%%%%%%%%%%%%%%%%
%\section{Solving Semilinear problems using a sequence of Approximate Problems}\label{Example1}
\section{Solution Construction using a Sequence of Approximate Problems}\label{Example1}
 If $\Om \subset \mathbb{R}^3$, the 
Sobolev embedding theorem tells us that $H^{1}(\Om)$ will compactly embed 
into $L^p(\Om)$ for $1 \le p < 6$ and continuously embed for $1 \le p \le 6$. 
Given functions  
$u,v \in H^{1}(\Om)$, this upper bound on $p$ places a constraint on the 
values of $i$ that allow for the product
$u^iv$ to be integrable.
In particular, Sobolev embedding and standard H\"{o}lder inequalities imply
that this product will be integrable for arbitrary elements of
$H^{1}(\Om)$ only if $1 \le i \le 5$.
More generally, if $a \in L^{\infty}(\Om)$, the term
$au^5v$ will also be integrable.
However, if $a$ is an unbounded function in $L^p(\Om)$
for some $p \ge 1$, then this product is not necessarily integrable without 
some sort of {\em a priori} bounds on $a, u$, and $v$.
Therefore, the following problem does not have
a well-defined weak formulation in $H^{1}(\Om)$:
\begin{align}\label{eq1:24oct11}
-\Delta u +au^m&+bu^i = 0 \quad \text{in $\Om$},\\
&u = \rho \quad \text{on $\partial\Om$} \nonumber,
\end{align}
where $\Om \subset\subset \Om'$, $m \ge 5$, $1 \le i \le 4$ are in $\mathbb{N}$, $\rho \in H^{1}(\Om')$, $b \in L^{\infty}(\Om')$ 
and $a \in L^p(\Om')$ for $\frac65 \le p < \infty$.

The objective of this section is to find a solution to the above problem.  
In order to solve \eqref{eq1:24oct11}, we solve a sequence of approximate, 
smooth problems and use a compactness argument to obtain a convergent 
subsequence.  
We first define necessary 
notation and then present the statements of two theorems that will be necessary for
our discussion in this section.  Then we prove the existence of a solution to
\eqref{eq1:24oct11}.
Finally, we show that if a solution exists, then under certain
conditions we can construct a net of problems whose solutions converge to the 
given solution.
    
\subsection{Overview of Spaces and Results for the Critical Exponent Problem}\label{prelim1}

For the remainder of the paper, for a fixed domain 
$\Om\subset \mathbb{R}^n$, we denote
the standard Sobolev norms on $\Om$ by
\begin{align}\label{Sobolev}
&\|u\|_{L^p} = \left( \int_{\Om} |u(x)|^p~dx\right)^{\frac1p},\\
&\|u\|_{W^{k,p}} = \left(\sum_{i=1}^k \|D^iu\|^p_{L^p}\right)^{\frac1p} \nonumber.
\end{align}
For the special case that $p=2$, we let $H^k(\Om) = W^{k,2}(\Om)$ and $H^1_0(\Om)$
denote the functions in $H^1(\Om)$ that have trace zero.
Furthermore, let 
\begin{align}
&\text{ess sup $u$} = \hat{u},\\
&\text{ess inf $u$} = \check{u}.\nonumber
\end{align}

In our subsequent work we will also require regularity conditions 
on the domain $\Om$ and its boundary.  Therefore, we will need the following definition taken from \cite{GiTr77}:

\begin{definition}
A bounded domain $\Om \subset \mathbb{R}^n$ and its boundary are of class $C^{k,\alpha}$, $0 \le \alpha \le1$,
if for each $x_0 \in \partial\Om$ there is a ball $B(x_0)$ and a one-to-one mapping $\Psi$ of $B$ onto $D \subset \mathbb{R}^n$
such that:
\begin{enumerate}
\item $\Psi(B \cap \Om) \subset \mathbb{R}^n_+$,
\item $\Psi(B \cap \partial\Om) \subset \partial\mathbb{R}^n_+$,
\item $\Psi \in C^{k,\alpha}(B), ~~~ \Psi^{-1} \in C^{k,\alpha}(D). $
\end{enumerate}
\end{definition}
We say that a domain $\Om$ is of class $C^{\infty}$ if for a fixed 
$0 \le \alpha \le 1$ it is of class $C^{k,\alpha}$
for each $k \in \mathbb{N}$.
Additionally, for this section and the next we will require the following Theorem and Proposition:
\begin{theorem}
\label{thm1:24oct11}
Suppose $\Om \subset \mathbb{R}^n$ is a $C^{\infty}$ domain and assume
$f:\overline{\Om} \times \mathbb{R}^+\to \mathbb{R}$ 
is in
$C^{\infty}(\overline{\Om}\times\mathbb{R}^+)$ and $\rho \in C^{\infty}(\overline{\Om})$.  Let $L$ be
an elliptic operator of the form
\begin{align}
Lu = -D_i(a^{ij}D_{j}u) + cu, \quad\text{and}\quad a^{ij}, c \in C^{\infty}(\overline{\Om}).
\end{align}
Suppose that there exist sub- and super-solutions 
$u_-:\overline{\Om} \to\mathbb{R}$ and $u_+:\overline{\Om}\to \mathbb{R}$ 
such that the following hold:
\begin{enumerate}
\item $u_-,u_+ \in C^{\infty}(\overline{\Om})$,
\item $0<u_-(x) < u_+(x) \hspace{3mm}\forall x\in \overline{\Om}.$ 
\end{enumerate}
Then there exists a solution $u\in C^{\infty}(\overline{\Om})$ to 
\begin{align}
\label{eq3june30-relabled-by-mjh}
Lu &= f(x,u) \hspace{3mm} \text{on $\Om$},\\
& u = \rho \quad \text{on $\partial\Om$}, \nonumber
\end{align}
such that $u_-(x)\le u(x) \le u_+(x). $
\end{theorem}  

\begin{proposition}
\label{prop2:24oct11}
%\mnote{Reformatted the statement of this prop. so that it only included what was relevant to following proposition.}
%\mnote{This looks good. -mjh}
Let $u$ be a solution to a semilinear equation of the form
\begin{align}\label{eq1:24nov11}
-\sum_{i,j}^ND_i (a^{ij} D_j u) &+ \sum_{i=1}^K b^iu^{n_i}=0 \text{  in  $\Om$}, \\
u  &= \rho, \quad \rho(x) > 0\hspace{2mm} \text{on} \hspace{2mm}\partial\Om \nonumber
\end{align}
where $a^{ij}, b^i$ and $\rho \in C^{\infty}(\ol{\Om}) $.   
Suppose that the semilinear operator in \eqref{eq1:24nov11} has the property that
$n_i>0$ for all 
$1 \le i \le K$.  Let $n_K$ be the largest positive exponent and 
suppose that $b^K(x)>0$ in $\overline{\Om}$. 
Define
\begin{align}\label{eq3:24oct11}
&\beta' = \inf_{c\in\mathbb{R}} \left\{\sum_{i=1}^K\inf_{x\in\ol{\Omega}}b^i(x)y^{n_i}> 0 \hspace{3mm}\forall y\in (c,\infty)\right\},\\
&\beta = \max\{\beta', \sup_{x\in \partial\Om} \rho(x)\}.
\end{align}

\noindent
Then if $u \in H^1(\Om)$ is a positive weak solution to Eq. \eqref{eq1:24nov11},
it follows that $ 0 \le u \le \beta < \infty$.
\end{proposition}

For the proof of Theorem~\ref{thm1:24oct11}, see Section~\ref{subsolution}. A more detailed version of Proposition~\ref{prop2:24oct11} and its
proof can be found in Section~\ref{bounds1}.  Now that we have all of the tools we need, we shall now prove the existence of a solution
to a problem of the form \eqref{eq1:24oct11}. 

\subsection{Existence of a Solution to an Ill-Posed Critical Exponent Problem}\label{example2}
For the following discussion, let $\Om' \subset \mathbb{R}^3$ be an open and bounded domain
and assume that $\Om \subset\subset \Om'$ is also open and of $C^{1,\alpha}$-class.

Here we seek a weak solution $u\in H^{1}(\Om)$ to the problem
\begin{align}\label{eq2:19oct11}
-\Delta u +au^m&+bu^i =0 \quad \text{in $\Omega$},\\
u & = \rho \quad \text{on $\Om$}, \nonumber
\end{align}
where $m \ge 5$, $1 \le i \le 4$ are in $\mathbb{N}$,
\begin{align}\label{eq1:21oct11}
b \in L^{\infty}(\Om'), 
\quad
 a \in L^p(\Om'), 
\quad
\frac65 \le p <\infty, 
\quad
\rho \in H^{1}(\Om') ,
\end{align} 
and
\begin{align}\label{eq2:21oct11}
\check{a} >0, \quad \hat{b} <0, \quad \text{and} \quad \check{\rho} >0.
\end{align}

If our test function space is $H^{1}_0(\Om)$, Eq. \eqref{eq2:19oct11} is ill-posed due to the
term $au^m$.  The weak formulation of Eq. \eqref{eq2:19oct11} would contain the integral
$$
\int_{\Om} au^mv~dx,
$$
where $v\in H^{1}(\Om)$, $a\in L^p(\Om)$, and $u\in H^{1}(\Om)$.  For these choices of function spaces
this integral need not be finite.  We show that this problem does in fact have a weak solution by regularizing the coefficients
of our problem and solving a sequence of approximating problems.  We obtain the following proposition.

\begin{proposition}\label{prop1:8nov11}
%\mnote{Tweaked proof and assumptions of this prop. and generalized it so that $m\ge 5$ instead of $m=5$.}
%\mnote{Agree; this is better. -mjh}
The semilinear problem \eqref{eq2:19oct11} has a solution $u\in H^{1}(\Om)$ if $a, b,\text{and} ~\rho$ satisfy the conditions in \eqref{eq1:21oct11} and \eqref{eq2:21oct11}.
\end{proposition}
\begin{proof}
To determine a solution to \eqref{eq2:19oct11}, we consider the sequence of solutions to the 
approximate problems
\begin{align}\label{eq1:19oct11}
-\Delta u_n +a_n(u_n)^m+b_n(u_n)^i &=0 \quad \text{in $\Omega$},\\
u_n = \rho_n \quad \text{on $\partial\Om$}, \nonumber 
\end{align}
where $a_n= a\ast \phi_n$, $b_n= b\ast \phi_n$, and $\rho_n = \rho\ast \phi_n$ and
$\phi_n = n^3\phi(nx)$ is a positive mollifier where $\int \phi(x)~dx = 1$.
Given that $\phi$ is a positive mollifier, it is clear that for each $n\in \mathbb{N}$,
$$
\check{a}_{n}>0, ~~~\hat{b}_n < 0~~~~ \text{and} \quad \check{\rho}_n>0.
$$
We first verify that the sequence of problems \eqref{eq1:19oct11} has a solution
for each $n$.  To do this, we will utilize Theorem~\ref{thm1:24oct11} and Proposition~\ref{prop2:24oct11}. 
Let $\beta_n$ have the same properties as $\beta$ in Proposition \ref{prop2:24oct11}
for the sequence of problems \eqref{eq1:19oct11}.  
Then using the notation in Proposition \ref{prop2:24oct11}, we can write explicit expressions for
$ \beta$ for \eqref{eq2:19oct11} and $\beta_n$.  It is not hard
to show that 
$$
\beta = \max\left\{\left(-\frac{\check{b}}{\check{a}}\right)^{\frac{1}{m-i}}, \hat{\rho}\right\},
$$
and
$$
\beta_n'= \left(-\frac{\check{b}_n}{\check{a}_n}\right)^{\frac{1}{m-i}}, \quad \beta_n = \max\left\{\beta_n', \hat{\rho}_n\right\}.
$$
By Proposition~\ref{prop2:24oct11}, for each $n \in \mathbb{N}$, $\beta_n$ determines an 
{\em a priori} upper bound for the approximate problems. 
Furthermore, it is not difficult to see that for each $n\in\mathbb{N}$ that
$0 $ and $\beta_n$ are sub- and super-solutions for \eqref{eq1:19oct11}.
See Section~\ref{supersolution} and Theorem~\ref{thm2june27} for more details. 
Therefore Theorem~\ref{thm1:24oct11} implies that for $n$ sufficiently large $0 \le u_n\le \beta_n$ is a solution in $ C^{\infty}(\overline{\Om})$ to \eqref{eq1:19oct11}
given that $\rho_n, a_n, b_n \in C^{\infty}(\overline{\Om})$ for $n$ sufficiently large.

Now observe that for each $n\in \mathbb{N}$, $\beta_n \le \beta$, which follows from the fact that
\begin{align}\label{eq1:22oct11}
-b_{n}(x) = \int (-b(y)) \phi_{n}(x-y)~dy \le \int (-\check{b})\phi_{n}(x-y)  = -\check{b},
\end{align}
and $a_{n}(x) \ge \check{a}$, which is verified by a similar calculation.  Therefore,  
by standard $L^p$ elliptic regularity theory 
\begin{align}
\|u_n\|_{W^{2,p}} \le& C( \|-a_n(u_n)^m-b_n(u_n)^i\|_{L^p} + \|u_n\|_{L^p}) \\
\le &C(\beta_n^m\|a_n\|_{L^p}+\beta_n^i\|b_n\|_{L^p}+\beta_n)< M < \infty, \nonumber
\end{align}
where $M$ is independent of $n$ given that $\beta_n \le \beta$, $a_n \to a$ in $L^p$, $b_n \to b$ in $L^p$.  
Because $p > \frac65$ and $\Om$ is of $C^{1,\alpha}$-class, $W^{2,p}(\Om)$ embeds compactly into $H^{1}(\Om)$. 
Therefore, there exists a convergent subsequence $u_{n_j} \to u $ in $H^{1}(\Om)$. 
We claim now that $u$ satisfies the following two properties:
\begin{enumerate}
\item $0 \le u \le \beta$ almost everywhere,
\item $u$ weakly solves \eqref{eq2:19oct11}.
\end{enumerate}
The inequality $0 \le u \le \beta~a.e.$ follows from the fact the $u_{n_j} \to u$ in $H^{1}(\Om)$ and 
$$
0  \le u_{n_j} \le \beta_{n_j} \le \beta \quad \text{for each $j \in \mathbb{N}$}.
$$
Indeed, if we assume that $u>\beta$ on some set of nonzero measure, then for some $n$ the set $A_n =\{x\in \Om: u(x) > \beta+\frac1n\}$ has positive measure.
Then for all $j\in \mathbb{N}$, we have that
$$
\int |u_{n_j} - u|^2~dx \ge \int_{A_n} |u_{n_j}-u|^2~dx \ge \frac{1}{n^2} \mu(A_n) > 0.
$$
But this clearly contradicts the fact that $u_{n_j} \to u$ in $H^{1}(\Om)$.  A similar argument 
shows that $ u \ge 0, ~a.e$ in $\Omega$.

Finally, we want to show that $u$ weakly solves \eqref{eq2:19oct11}.  Let $\e >0$.  Then for any $v\in H^{1}_0(\Om)$ we have
that
\begin{align}
& \left|\int \left( \nabla u\cdot\nabla v + au^mv + bu^i v\right)~ dx\right|
= \left|\int \left(\nabla u\cdot\nabla v + au^mv + bu^i v\right)~ dx \right. \\
& \qquad \qquad \qquad \qquad \left.
  - \int \left(\nabla u_{n_j}\cdot\nabla v + a_{n_j}(u_{n_j})^mv + b_{n_j}(u_{n_j})^i v\right)~ dx\right|, \nonumber
\end{align}
given that $u_{n_j}$ solves \eqref{eq1:19oct11}.
Then expanding the second line of the above equation we find that
\begin{align}
\label{eq3:19oct11}
&\left|\int \nabla u\cdot\nabla v + au^mv + bu^i v~ dx\right| \\
& \qquad \le \int \left|\nabla u\cdot\nabla v- \nabla u_{n_j}\cdot\nabla v\right|~dx+\int\left| au^mv-  a_{n_j}(u_{n_j})^mv\right|~dx \nonumber \\
& \qquad \qquad + \int\left|bu^i v- b_{n_j}(u_{n_j})^i v\right|~ dx  \\
& \qquad \le \int \left|\nabla u\cdot\nabla v- \nabla u_{n_j}\cdot\nabla v\right|~dx +\int\left| au^mv-a(u_{n_j})^mv\right|~dx  \nonumber \\
& \qquad \qquad +\int\left|a(u_{n_j})^mv-  a_{n_j}(u_{n_j})^mv\right|~dx+ \int\left|bu^i v-b(u_{n_j})^iv\right|~dx \nonumber \\
& \qquad \qquad + \int \left|b(u_{n_j})^iv- b_{n_j}(u_{n_j})^i v\right|~ dx.
\label{eq3:19oct11-b}
\end{align}
Every term in 
% the bottom two lines of 
\eqref{eq3:19oct11-b} tends to $0$ given that $u_{n_j} \to u$ in $H^{1}(\Om)$, 
$a_{n_j} \to a$ in $L^p(\Om)$, $b_{n_j} \to b$ in $L^p(\Om)$ and $0 \le u \le \beta$.  To show that the expression
$$
\int\left| au^mv-a(u_{n_j})^mv\right|~dx \to 0,
$$
we apply H$\ddot{\text{o}}$lder's inequality to obtain
$$
\int\left| au^mv-a(u_{n_j})^mv\right|~dx  \le \| a\|_{L^{\frac65}} \|u^mv-u_{n_j}^mv\|_{L^6}.
$$
Given that $u_{n_j} \to u $ in $H^{1}(\Om)$, $u_{n_j} \to u~~a.e$, where we pass to a subsequence if 
necessary.  Therefore $u_{n_j}^mv \to u^mv~~a.e$.  Finally, we observe that
$$
|u^m-u_{n_j}^m|^6|v|^6 \le 64\beta^{6m}|v|^6,
$$
and given that $v \in H^{1}(\Om)$ and $\Om$ is bounded, the Dominated Convergence Theorem implies
that $\|u^mv-u_{n_j}^mv\|_{L^6} \to 0$.  Therefore
$$
\int\left| au^mv-a(u_{n_j})^mv\right|~dx \to 0.
$$
We apply a similar argument to show that the lower order terms in Eq. \eqref{eq3:19oct11}
converge to zero and conclude that $u$ is a weak solution to \eqref{eq2:19oct11}.
\end{proof}

% \mnote{Deleted proposition 2.5.  Didn't really add anything new.  Simply switched from sequences to nets.}
% \mnote{Agree with you here. -mjh}

%%%%%%%%%%%%%%%%%%%%%%%%%%%%%%%%%%%%%%%%%%%%%%%%%%%%%%%%%%%%%%%%%%%%%%%%%%%%%%
\section{Preliminary Material: H\"{o}lder Spaces and Colombeau Algebras}
\label{prelim}

We now begin to develop the Colombeau Algebra framework that will be used 
to solve \eqref{problem}.
We first define H\"{o}lder Spaces and state precise versions of the
classical Schauder estimates given in \cite{MP06}.
The definition of the Colombeau Algebra in which we will be working and these 
classical elliptic regularity estimates make these spaces the most natural 
choice in which to do our analysis.
Therefore we will work almost exclusively with H\"{o}lder spaces for the 
remainder of the paper.
Following our discussion of function spaces, we define the Colombeau algebra 
in which we will work and then formulate an elliptic, semilinear problem 
in this space.

%%%%%%%%%%%%%%%%%%%%%%%%%%%%%%%%%%%%%%%
\subsection{Function Spaces and Norms}
\label{holder}
% \mnote{Maybe include the $ L^p$ estimates here as well?  -cm}
In this paper we will make frequent use of Schauder estimates on H\"{o}lder spaces
defined on an open set $\Om \subset \mathbb{R}^n$.
Here we give notation for the H\"{o}lder norms and then state the regularity estimates that will be used.  

All notation and results are taken from~\cite{GiTr77}.
Assume that $\Om \subset \mathbb{R}^n$ is open, connected and bounded.
Then define the following norms and seminorms:
\begin{align}
	&[u]_{\alpha;\Om} = \sup_{\stackrel{x,y \in \Om}{x \ne y}} \frac{|u(x)-u(y)|}{|x-y|^{\alpha}},\\
	&[u]_{k,0;\Om} = \sup_{|\beta|=k} \sup_{x\in\ol{\Om}}|D^{\beta}u|,\\
	&[u]_{k,\alpha;\Om} = \sup_{|\beta|=k}[D^{\beta}u]_{\alpha;\Om},\\
	&\|u\|_{C^k(\overline{\Om})} =|u|_{k;\Om}= \sum_{j=0}^{k}[u]_{j,0;\Om},\\
	&\|u\|_{C^{k,\alpha}(\overline{\Om})} =|u|_{k,\alpha;\Om} = |u|_{k;\Om}+[u]_{k,\alpha;\Om}.
\end{align}
We interpret $C^{k,\alpha}(\overline{\Om})$ as the subspace of functions $f \in C^k(\overline{\Om})$ such that $f^{(k)}$ is $\alpha$-H\"{o}lder continuous.
Also, we view the subspace $C^{k,\alpha}(\Om)$ as the subspace of functions $f\in C^k(\Om)$ such that $f^{(k)}$ is locally $\alpha-$H\"{o}lder continuous
(over compact sets $K\subset\subset\Om$).

Now we consider the equation
\begin{align}
\label{eq4june29}
Lu = a^{ij}D_{ij}u + b^i uD_iu +cu &=f \quad \text{in $\Om$},\\
u &= \rho \quad \text{on $\partial\Om$},
\end{align}
where $L$ is a strictly elliptic operator satisfying
$$
a^{ij} = a^{ji} \quad \text{and} \quad a^{ij}(x)\xi_i\xi_j \ge \lambda|\xi|^2, \quad x\in \Om, \quad \xi \in \mathbb{R}^n.
$$
The following regularity theorems can be found in~\cite{GiTr77} and~\cite{MP06}.  See~\cite{GiTr77} for proofs.
Note that the constant $C$ in the following theorems has no dependence on $\Lambda$ or $\lambda$.  
\begin{theorem} 
\label{thm1june30}
Assume that $\Om$ is a $C^{2,\alpha}$-class domain in $\mathbb{R}^n$ and that $u\in C^{2,\alpha}(\overline{\Om})$ is
a solution~\eqref{eq4june29}, where $f\in C^{\alpha}(\overline{\Om})$ and $\rho\in C^{2,\alpha}(\overline{\Om})$.  Additionally assume
that 
$$
|a^{ij}|_{0,\alpha;\Om}, |b^i|_{0,\alpha;\Om}, |c|_{0,\alpha;\Om} \le \Lambda.
$$
Then there exists
$C>0$ such that
$$
|u|_{2,\alpha;\Om} \le C\left(\frac{\Lambda}{\lambda}\right)^3(|u|_{0;\Om}+|\rho|_{2,\alpha;\Om}+|f|_{0,\alpha;\Om}).
$$
\end{theorem}
\noindent
This theorem can then be extended to higher order derivatives by repeatedly applying Theorem~\ref{thm1june30}.  
See~\cite{MP06} for details.  We summarize this result
in the next theorem.
% \mnote{ We might not need this Theorem.  -cm}
\begin{theorem}
\label{thm2june30}
Let $\Om$ be a $C^{k+2,\alpha}$-class domain and $u\in C^2({\Om})\cap C^0(\overline{\Om})$ be a solution of~\eqref{eq4june29}, where $f\in C^{k,\alpha}(\overline{\Om})$ and $\rho\in C^{k+2,\alpha}(\overline{\Om})$.
Additionally assume that
$$
|a^{ij}|_{k,\alpha;\Om}, |b^i|_{k,\alpha;\Om}, |c|_{k,\alpha;\Om} \le \Lambda.
$$
Then $u\in C^{k+2,\alpha;\Om}(\overline{\Om})$ and 
$$
|u|_{k+2,\alpha;\Om} \le C^{k+1}\left(\frac{\Lambda}{\lambda}\right)^{3(k+1)}(|u|_{0;\Om}+|\rho|_{k+2,\alpha;\Om}+|f|_{k,\alpha;\Om}),
$$
where $C$ is the constant from Theorem~\ref{thm1june30}.
\end{theorem}

%%%%%%%%%%%%%%%%%%%%%%%%%%%%%%%%%%%%%%%
\subsection{Colombeau Algebras}
\label{Colombeau}

Now that we have defined the basic function spaces that we will be working 
with and stated the regularity theorems that will be required to obtain 
necessary growth estimates, we are ready to define the Colombeau algebra with which we 
will be working and formulate our problem in this algebra.

Let $V$ be a topological vector space whose topology is given by an increasing
family of seminorms $\mu_k$.
That is, for $u \in V$, $\mu_{i}(u) \le \mu_{j}(u)$ if
$i \le j$.  Then letting $I = (0,1]$, we define the following:
\begin{align}
\label{eq1july7}
\mathcal{E}_V &= (V)^I \hspace{4mm} \text{where $u\in \mathcal{E}_V$ is a net $(u_{\e})$ of elements in $V$ with $\e \in (0,1]$}, \\
\mathcal{E}_{M,V} &= \{ (u_{\e}) \in \mathcal{E}_V \hspace{2mm} | \hspace{2mm} \forall k \in \mathbb{N} \hspace{2mm} \exists a\in \mathbb{R} \hspace{2mm} : \hspace{2mm} \mu_k(u_{\e}) = \mathcal{O}(\e^a) \hspace{2mm} \text{as } \e \to 0\}, \\
\mathcal{N}_V &= \{ (u_{\e}) \in \mathcal{E}_{V,M} \hspace{2mm} | \hspace{2mm} \forall k \in \mathbb{N} \hspace{2mm} \forall a\in \mathbb{R} \hspace{2mm} : \hspace{2mm} \mu_k(u_{\e}) = \mathcal{O}(\e^a) \hspace{2mm} \text{as } \e \to 0\} \label{eq1:15nov11}.
\end{align}
Then the polynomial generalized extension of $V$ is formed
by considering the quotient $\mathcal{G}_V = \mathcal{E}_{M,V}/\mathcal{N}_V$.

We now give a few examples of generalized extensions.  See \cite{MP06,GKOS01} for a more detailed discussion.

\begin{definition}\label{eq1:3nov11}
%\mnote{Added example to clarify definition.}
%\mnote{Good idea to include this. -mjh}
If $V = \mathbb{C}$, $r \in \mathbb{C}$, $\mu_k(r) =|r|$, then one obtains $\overline{\mathbb{C}}$, the ring of generalized constants. This
ring contains all nets of complex numbers that grow no faster than a polynomial in $\e^{-1}$ as $\e \to 0$.  For example,
$(\textnormal{e}^{\e^{-1}}) \not\in \mathbb{\ol{C}}$ given that this net grows exponentially in $\e^{-1}$ as $\e \to 0$. 
\end{definition}

\begin{definition}\label{ex1:27oct11}
Let $\Om \subset \mathbb{R}^n$ 
be an open set, $U_k \subset\subset \Om$ an exhaustive sequence of compact sets and $\alpha \in \mathbb{N}^n_0$ a multi-index. Then if 
$$V = C^{\infty}(\Om), \hspace{2mm}f \in C^{\infty}(\Om), \hspace{4mm} \mu_{k}(f)= \sup\{|D^{\alpha}f| \hspace{1mm} : \hspace{1mm} x \in U_k,\hspace{2mm}|\alpha| \le k\},$$
one obtains $\mathcal{G}^s(\Om)$, the simplified Colombeau Algebra.
\end{definition}
\begin{definition}\label{def1:3feb13}
If $V=C^{\infty}(\overline{\Om})$, where $\Om\subset \mathbb{R}^n$ is bounded and 
$$\mu_k(f)=\sup\{|D^{\alpha} f |:|\alpha|\le k,\hspace{2mm} x\in\overline{\Omega}\},$$
we denote the generalized extension by $\mathcal{G}(\overline{\Om})$.  The
set $\mathcal{E}_{M,C^{\infty}(\overline{\Om})}$ will be denoted by $\mathcal{E}_M(\overline{\Om})$ 
and be referred to as the space of moderate elements.  The set $\mathcal{N}_{C^{\infty}(\overline{\Om})}$ will
be denoted by $\mathcal{N}(\overline{\Om})$ and will be referred to as the space of null elements.
\end{definition}
Both $\mathcal{G}^s(\Om)$ and $\overline{\mathbb{C}}$ were developed by Colombeau and laid the basis for the more general construction
described in \eqref{eq1july7}-\eqref{eq1:15nov11}.  See~\cite{JCo84} for more details.  
As in~\cite{MP06}, for the purposes of this paper
we are concerned with $\mathcal{G}(\overline{\Omega})$ given that we are interested in solving the Dirichlet problem
and require a well-defined boundary value.  
If $(u_{\e}) \in \mathcal{E}_M(\overline{\Om})$ is a representative of an element $u \in \mathcal{G}(\overline{\Om})$, 
we shall write $u = [(u_{\e})]$ to indicate that $u$ is the equivalence class of $(u_{\e})$.  At times we will drop the parentheses
and simply write $[u_{\e}]$.  Addition and 
multiplication of elements in $\mathcal{G}(\overline{\Om})$ is defined in terms of addition and multiplication of
representatives.  That is, if $u = [(u_{\e})]$ and $v = [(v_{\e})]$, then $uv = [(u_{\e}v_{\e})]$ and $u+v = [(u_{\e}+v_{\e})]$.
Derivations are defined for $u = [(u_{\e})] \in \mathcal{G}(\overline{\Om})$ by $\partial_{x_i} u = [(\partial_{x_i}u_{\e})]$.
\begin{theorem}
With the above definitions of addition, multiplication and differentiation, $\mathcal{G}(\overline{\Om})$ is a associative, commutative, differential algebra.
\end{theorem}
\begin{proof}
This follows from the fact component-wise addition, multiplication, and differentiation makes $V^I = (C^{\infty}(\overline{\Om}))^I$ into a differential
algebra.  By design, $\mathcal{E}_M(\overline{\Om})$ is the largest sub-algebra of $(C^{\infty}(\overline{\Om}))^I$ that contains $\mathcal{N}(\overline{\Om})$
as an ideal.  Therefore $\mathcal{G}(\overline{\Om})$ is a differential algebra as well.  See~\cite{GKOS01}.
\end{proof}
Now that we have given the basic definition of a Colombeau algebra, we can discuss how distributions can be embedded
into a space of this type.

\subsection{Embedding Schwartz Distributions into Colombeau Algebras}\label{embed}
%\mnote{Completely reworked this entire section.  Changed definition of embedding.}
%\mnote{Reads well now. -mjh}
While the algebras defined above are somewhat unwieldy, these spaces
are well suited for analyzing problems with distributional data.
The primary reason
for this is that for a given open set $\Om\subset \mathbb{R}^n$, the Schwartz distributions $\mathcal{D}'(\Om)$ can be 
linearly embedded into $\cG^s(\Om)$.
This allows one to define an {\em extrinsic} notion of distributional
multiplication that is consistent with the pointwise product of 
$C^{\infty}(\Om)$ functions.     
Here we briefly discuss the method use to embed $\mathcal{D}'(\Om)$ into 
$\mathcal{G}^s(\Om)$.  Given that we will primarily be working with the
generalized extension $\cG(\ol{\Om})$ defined in \eqref{def1:3feb13}, we will then discuss how 
to embed certain subsets of $\cD'(\Om)$into $\cG(\ol{\Om})$.

We begin by recalling the definitions of the spaces that will be relevant to 
our discussion.
The Schwartz distributions on an open set $\Om\subset \mathbb{R}^n$ are 
denoted $\mathcal{D}'(\Om)$ and are defined to be the dual of 
$\mathcal{D}(\Om)$, the space of $C^{\infty}(\Om)$ functions
with support contained in $\Om$.
For a given $\varphi \in \mathcal{D}(\Om)$ and $T \in \mathcal{D}'(\Om)$, 
the action of $T$ on $\varphi$ will be denoted by 
$\left\langle T, \varphi \right\rangle $.  We let $\cE'(\Om)\subset \cD'(\Om)$ denote the
denote the space of compactly supported distributions.
Finally, we define the space of Schwartz functions $\cS(\mathbb{R}^n)$ by
\begin{align}
\cS(\mathbb{R}^n) = \left\{ f \in C^{\infty}(\mathbb{R}^n)~|~\|f\|_{\alpha,\beta} < \infty, ~~ \forall \alpha,\beta \right\}, \quad \|f\|_{\alpha,\beta} = \sup_{x\in\mathbb{R}^n}|x^{\alpha} D^{\beta}f(x)|,
\end{align}
where $\alpha, \beta$ are multi-indices.

Let $\varphi \in \cD(\mathbb{R}^n)$ satisfy
\begin{align}\label{eq1:5feb13}
\varphi(x) \ge 0, \quad \int_{\mathbb{R}^n} \varphi(x)~dx = 1, \quad \lim_{\e\to 0} \varphi_{\e}(x) = \lim_{\e to 0} \e^{-n}\varphi\left(\frac{x}{\e}\right) \to \delta(x).
\end{align}
So $\varphi(x)$ is a standard, positive mollifying function.  To construct our embedding, we will also require
another function with more restrictive properties. 
Let $\psi \in \mathcal{S}(\mathbb{R}^n)$ be a function such that $\psi \equiv 1$ on some neighborhood of $0$.  
Then define $\phi \in \mathcal{S}(\mathbb{R}^n)$
by $\phi = \mathcal{F}^{-1}[\psi]$, the inverse Fourier transform of $\psi$. 
It is easy to see that 
\begin{align}\label{eq4:28oct11}
\int_{\mathbb{R}^n} \phi \hspace{1mm}dx = 1 \hspace{2mm}  \text{ and} \hspace{2mm} \int_{\mathbb{R}^n} x^{\alpha}\phi \hspace{1mm}dx = 0 \quad \forall |\alpha| \ge 1.
\end{align}
Let  $\phi_{\e} = \e^{-n}\phi(\frac{x}{\e})$.  

The properties of $\phi$ specified in \eqref{eq4:28oct11} are extremely important.
By convolving with the function $\phi_{\e}$ and using the sheaf properties of the space
$\cG^s(\Om)$, one is able to construct a linear embedding
\begin{align}\label{eq3:28oct11}
&\it{i} : \mathcal{D}'(\Om) \to \mathcal{G}^s(\Om)
\end{align}
See \cite{GKOS01} for details.  An important property of this
embedding is that for any ${f, g \in C^{\infty}(\Om)}$, ${\it{i}(fg) = \it{i}(f)\it{i}(g)}$.
Therefore, multiplication in $\cG^s(\Om)$ is an extension of point-wise multiplication   
of $C^{\infty}$ functions.

We now discuss a method of embedding certain subsets of $\cD'(\Om)$ into $\cG(\overline{\Om})$.
The reason that we must restrict our embedding to certain subsets of $\cD'(\Om)$ is that our generalized
extension $\cG(\ol{\Om})$ is defined on the closed set $\ol{\Om}$.  Colombeau algebras of this form
no longer have a sheaf structure and so we can no longer take advantage of the general embedding
\eqref{eq3:28oct11} constructed in \cite{GKOS01}.  However, as in the case with the embedding constructed
in \cite{GKOS01}, our main tool for constructing our embedding will be convolution with the function $\phi$
satisfying the properties in \eqref{eq4:28oct11}.

The most natural way to associate a given element $u \in \cD'(\Om)$ with a net of $C^{\infty}(\ol{\Om})$ functions
is by mollifying $u$ with a function like $\varphi$ defined in \eqref{eq1:5feb13}.  But in order for our embedding to preserve point-wise multiplication of
$C^{\infty}(\ol{\Om})$ functions, we need the functions that we convolve with $u$ to have the same properties as $\phi_{\e}$ defined in \eqref{eq4:28oct11}.
However, $\phi_{\e} \in \cS(\mathbb{R}^n)$ for each $\e\in (0,1]$, so convolution with an arbitrary element $u \in \cD'(\Om)$ is not well-defined. 
This is where the sheaf properties of $\cG(\Om)$ are instrumental in constructing the embedding \eqref{eq3:28oct11}.  However, we no
longer have this option, and therefore focus on finding a subsets of $\cD'(\Om)$ for which the convolution is defined.

Given $f\in C^{\infty}(\ol{\Om})$, we again observe that the convolution $(f\ast \phi_{\e})$ is not well-defined for all $x \in \ol{\Om}$, $\e \in (0,1]$.  This
follows because $f$ has no value outside of $\ol{\Om}$.  What we seek is a way to extend $C^{\infty}(\ol{\Om})$ functions to $C^{\infty}(\mathbb{R}^n)$ functions, and more generally, elements of $\cD'(\Om)$ to $\cD'(\mathbb{R}^n)$, so that the convolution has meaning.  We note that it is not possible to extend an arbitrary element of $\cD'(\Om)$ to $\cD'(\mathbb{R}^n)$, so we will restrict ourselves to a subspace of $\cD'(\Om)$.  The following theorem taken from \cite{RA75} will provide us with a large subspace of $\cD'(\Om)$ that we can
extend.

\begin{theorem}\label{thm1:4feb13}
Suppose that $\Om' \subset\subset \Om$, and that $\Om$ is bounded and of $C^{\infty}$-class.  Then there exists a total extension operator, which has
the property that for each $0 \le k\le \infty$, $1 \le p \le \infty$
\begin{align}
&E: W^{k,p}(\Om)  \to W^{k,p}(\mathbb{R}^n),\\
&E(u)|_{\Om} = u,\nonumber
\end{align}
and 
\begin{align}\label{eq2:6feb13}
\|Eu\|_{W^{k,p}(\mathbb{R}^n)} \le C(n,p)\|u\|_{W^{k,p}(\Om)}.
\end{align}
\smallskip

\noindent
Moreover, $E$ can be extended to $\cE'(\Om') \subset \cE'(\Om)$ so that if $u \in \cE'(\Om')$,
\begin{align}
&E(u)|_{\Om} = u \\
&E(u)|_{\Om^c} = 0. \nonumber
\end{align}
\end{theorem}
\begin{proof}
Let $ \alpha = d(\Om',\partial\Om) > 0$.  Given that $\Om$ is of $C^{\infty}$-class, we may cover $\partial\Om$ with finitely many balls of radius $\alpha/2$ (or smaller if necessary)
that are $C^{\infty}$-diffeomorphic with some subset of ${B_1(0) \cap \mathbb{R}^n_+}$.
As in the proof of Theorem 4.28 in \cite{RA75}, we may use these neighborhoods to construct a total extension
operator which has the property that for every $0 \le k \le \infty$, $1 \le p \le \infty$,
\begin{align}
&E: W^{k,p} \to W^{k,p}(\mathbb{R}^n),\\
&E(f)|_{\Om} = f.
\end{align}

We can extend this extension operator to $\cE'(\Om')$.  It is well know that for 
any $u \in W^{k,p}(\Om)$, there exists an approximating net $\{u_{\e}\} \subset C^{\infty}(\ol{\Om})$
such that $u_{\e} \to u$ in $W^{k,p}(\Om)$.  See the Global Approximation Theorem in \cite{LE98}.
Using the same argument as in the proof of this theorem, we can
obtain an approximating net of $C^{\infty}(\ol{\Om})$ functions for $u\in \cE'(\Om')$.
We have that $u = \partial^{\alpha} f$ for some continuous $f$ with support in an arbitrary neighborhood of $\text{supp}(u)$.
By shifting the argument of $f$ and mollifying up to the boundary in each of the balls covering $\partial\Om$ defined above,
and then applying a partition of unity argument, we obtain a net $\{u_{\e} \} \subset C^{\infty}(\ol{\Om})$ such that $u_{\e} \to u$ in $\cD'(\Om)$.
Furthermore, for this net $\{u_{\e}\}$ there exists $\e_0 \in (0,1)$ for which $u_{\e} \equiv 0$ on $\Om\cap \ol{\Om'}^c$~~ if $0<\e < \e_0$.  
For a given $u \in \cE'(\Om')$, we let $\{u_{\e}\}$ denote this approximating net and we define
$$
E(u) = \lim_{\e\to 0} E(u_{\e}).
$$
Based on the properties of $u_{\e}$, this extension will extend $u$ by zero outside of $\Om$.  We note that for
$u\in W^{k,p}(\Om) \cap \cE'(\Om')$, this definition of extension on $\cE'(\Om')$ will be consistent with the
extension on $W^{k,p}(\Om)$ given the properties of $E$ in \eqref{eq2:6feb13}. 
\end{proof}

We now define the following subspace of $\cD'(\Om)$ that we will embed into $\cG(\ol{\Om})$.  Fix an open subset of $\Om' \subset\subset \Om$,
where $\Om$ is of $C^{\infty}$-class.  Let 
%\mnote{The following definition may require some more thought.}
%\mnoteR{I am still looking at this one technical issue. -mjh}
\begin{align}
\cF'(\Om) = \cE'(\Om') + \left( \bigcup_{0 \le k \le \infty, 1 \le p \le \infty} W^{k,p}(\Om) \right),
\end{align}
\noindent
where the above notation indicates the subspace formed by the sum of $\cE'(\Om)$ and the union of the Sobolev spaces as subspaces of
$\cD'(\Om)$. 

\begin{theorem}\label{thm1:5feb13}
Let $E$ be the extension operator defined in Theorem~\ref{thm1:4feb13} and let $\phi_{\e} \in \cS(\mathbb{R}^n)$ be the 
net of functions defined in \eqref{eq4:28oct11}.  Then the map
\begin{align}\label{eq1:6feb13}
&\it{i}: \cF'(\Om) \to \cG(\ol{\Om}),\\
&\it{i}(u) = \left.(E(u)\ast \phi_{\e})\right|_{\ol{\Om}} + \cN(\ol{\Om}) \nonumber,
\end{align}  
is a linear embedding of $\cF'(\Om)$ into $\cG(\ol{\Om})$.
\end{theorem}
\begin{proof}
By the linearity of the extension operator $E$, we observe that for any $u \in \cE'(\Om')$ and
$v\in W^{k,p}(\Om)$, $0 \le k \le \infty$, $1 \le p \le \infty$, $E(u+v)$ is well defined and unique
in the distributional sense.  Therefore, for any element $u \in \cF'(\Om)$, $E(u) \in \cD'(\mathbb{R}^n)$ is 
unique and $\it{i}$ is well-defined.  For any $u \in W^{k,p}(\Om)$ and multi-index $\alpha$, we have that
\begin{align}\label{eq3:6feb13}
\partial^{\alpha}(&E(u)\ast \phi_{\e}) = \\
&\int E(u)(x)\partial^{\alpha}\phi_{\e}(x-y)~dy = \int E(u)(x-\e y)\e^{-|\alpha|}\partial^{\alpha}\phi(y)~dy = \cO(\e^{-|\alpha |}), \nonumber
\end{align}  
given that $E(u) \in W^{k,p}(\mathbb{R}^n)$ and $\phi \in \cS(\mathbb{R}^n)$.  So $\it{i}(u) \in \cG(\ol{\Om})$.  A similar argument can be used
to show that $\it{i}(v) \in \cG(\ol{\Om})$ if $v \in \cE'(\Om')$.  By linearity, $\it{i}(u) \in \cG(\ol{\Om})$ for any $u \in \cF'(\Om)$.  Now we only need to 
show that $\it{i}$ is injective.  Suppose that $\it{i}(u) \in \cN(\ol{\Om})$.  Then $E(u) \ast \phi_{\e}(x) \to 0$ uniformly on $\ol{\Om}$.  Therefore,
for any $\psi \in \cD(\Om)$, 
$$
\langle u, \psi \rangle = \lim_{\e \to 0} \langle u\ast \phi_{\e}, \psi \rangle = \lim_{\e \to 0} \langle E(u) \ast \phi_{\e}, \psi\rangle = 0.
$$ 
So $u \equiv 0$ in $\cD'(\Om)$ and $\it{i}$ is injective.
\end{proof}

The embedding $\it{i}$ has the important property that it preserves point-wise multiplication of $C^{\infty}(\ol{\Om})$ functions.
We prove this by following the argument in \cite{GKOS01}.
We first observe that we may embed $f\in C^{\infty}(\ol{\Om})$ into $\cG(\ol{\Om})$ by the map
\begin{align}\label{eq4:6feb13}
\sigma: C^{\infty}(\ol{\Om}) \to \cG(\ol{\Om}),\\
\sigma(f) = (f)_{\e} + \cN(\ol{\Om})\nonumber
\end{align}
where $(f)_{\e}$ the constant net such that $f_{\e} = f $ for all $\e \in (0,1]$.

\begin{proposition}\label{prop1:6feb13}
The embedding $\it{i}$ has the property that $\it{i}|_{C^{\infty}(\ol{\Om}) }= \sigma$.
\end{proposition}
\begin{proof}
This follows from the proof of Proposition 1.2.11 in \cite{GKOS01} and the fact that if
$u \in C^{\infty}(\ol{\Om})$, then $E(u) \in C^{\infty}(\mathbb{R}^n)$ and $E(u)|_{\ol{\Om}} = u$.
\end{proof}

Proposition~\ref{prop1:6feb13} allows us to conclude that $\it{i}$ preserves point-wise multiplication
of $C^{\infty}(\ol{\Om})$ functions.  Indeed, if $f,g \in C^{\infty}(\ol{\Om})$, then
$$
\it{i}(fg) = \sigma(fg) = \sigma(f)\sigma(g) = \it{i}(f)\it{i}(g).
$$
\noindent
Now that we have a means of embedding a rather large class of elements of $\cD'(\Om)$ into
$\cG(\ol{\Om})$ that are useful for solving PDE, we can begin to formulate what a semilinear problem
in $\cG(\ol{\Om})$ looks like.

%%%%%%%%%%%%%%%%%%%%%%%%%%%%%%%%%%%%%%%
\subsection{Nets of Semilinear Differential Operators}
\label{netsofproblems}
We begin by defining a semilinear differential operator on $\mathcal{G}(\ol{\Om})$.  Our construction
strongly resembles the construction by Mitrovic and Pilipovic in \cite{MP06}.
For $\epsilon<1$,  if $(a^{ij}_{\epsilon})$, $(b^i_{\epsilon}) \in \mathcal{E}_M(\overline{\Om})$, we obtain a net of operators 
by defining $A_{\e}$ to be
$$
A_{\epsilon}u_{\e} = - D_i( a_{\epsilon}^{ij}D_j u) + \sum_i^K b^i_{\epsilon}u^{n_i} = - a^{ij}_{\e} D_iD_ju_{\e}- (D_ia^{ij}_{\e})(D_j u_{\e})+\sum_{i=1}^K b^i_{\e}(u_{\e})^{n_i},
$$
where $n_i \in \mathbb{Z}$.  Under certain conditions, we can view a net of operators of the above form as an operator on $\mathcal{G}(\ol{\Om})$.  Here we determine these conditions, which will guarantee that this net of operators is a well-defined operator on $\mathcal{G}(\overline{\Om})$. 

Given an element $u$ in $\mathcal{G}(\ol{\Om})$, we first need to ensure that $(A_{\e}u_{\e}) \in \mathcal{E}_{M}(\ol{\Om})$.
Based on how derivations and multiplication are defined in $\mathcal{G}(\ol{\Om})$, the only serious obstacle to this 
is if $n_i <0$ for some $i \le K$.
Therefore, we must guarantee that the
element $((u_{\e})^{n_i})$ is a well-defined representative in $\mathcal{G}(\overline{\Om})$ if $n_i <0$.  It suffices to ensure that $u =[(u_{\e})]$ has an inverse in $\mathcal{G}(\overline{\Om})$.  This is 
true if for each representative $(u_{\e})$ of $u$, there exists $\e_0 \in (0,1]$ and $m\in \mathbb{N}$ such that for all $\e \in (0,\e_0)$, $\inf_{x\in \overline{\Om}}|u_{\e}(x)| \ge C\e^m$.  See \cite{GKOS01} for more details.
So $u\in \mathcal{G}(\overline{\Om})$ must possess this property in order for the above operator to have any chance of being well-defined.  For the rest of this section we assume that $u$ satisfies this condition. 

Now suppose $(\overline{a}^{ij}_{\epsilon})$, $(\overline{b}^i_{\epsilon})$ in $\mathcal{E}_M(\overline{\Omega})$,
and let
$$
 \overline{A}_{\epsilon}u = -\sum_{i,j=1}^N D_i( \overline{a}_{\epsilon}^{ij}D_j u) + \sum_i^K \overline{b}^i_{\epsilon}u^{n_i} = \overline{a}^{ij}_{\e}D_iD_ju_{\e}-(D_i \overline{a}^{ij}_{\e})(D_ju_{\e})+\sum_{i=1}^K \overline{b}^i_{\e}(u_{\e})^{n_i}.
$$
We say that $(A_{\epsilon}) \sim (\overline{A}_{\epsilon})$ if $(a^{ij}_{\epsilon}-\overline{a}^{ij}_{\epsilon}),(b^i_{\epsilon}-\overline{b}^i_{\epsilon}) \in \mathcal{N}^s(\overline{\Omega})$.  Then $(A_{\epsilon})\sim (\overline{A}_{\epsilon})$ if and only if $(A_{\epsilon}u_{\epsilon}-\overline{A}_{\epsilon}u_{\epsilon})\in\mathcal{N}(\overline{\Omega})$ for all $(u_{\epsilon})\in\mathcal{E}_M(\overline{\Omega})$ due to the fact that the above operators are linear in $(a^{ij}_{\epsilon})$ and $(b^i_{\epsilon})$.

Let $\mathcal{A}$ be the family of nets of differential operators of the above form and define $\mathcal{A}_0=\mathcal{A}/\sim$.
Then for $A\in \mathcal{A}_0$ and $u\in \mathcal{E}_M(\overline{\Omega})$, define
$$
A:\mathcal{G}(\overline{\Omega})\to\mathcal{G}(\overline{\Omega}) \text{   by   } Au=[A_{\epsilon}u_{\epsilon}],
$$
where
\begin{align}
\label{eq1june26}
[A_{\epsilon}u_{\epsilon}]=[-a^{ij}_{\e}][D_iD_ju_{\epsilon}]+[-D_ia^{ij}_{\e}][D_j u_{\e}] + \sum_{i=1}^K[b^i_{\epsilon}][u^{n_i}_{\epsilon}].
\end{align}
Using this definition, $A\in \mathcal{A}_0$ is a well-defined operator on $\mathcal{G}(\overline{\Om})$.  We summarize this statement in the following proposition.

\begin{proposition}
\label{prop1july1}
$\mathcal{A}_0$ is a well-defined class of differential operators from $\mathcal{G}(\overline{\Om})$ to $\mathcal{G}(\overline{\Om})$.
\end{proposition}
\begin{proof}
Based on the construction of $\mathcal{A}_0$, it is clear that for a given representative
$(u_{\e})$ of $u \in \mathcal{G}(\overline{\Om})$, $(A_{\e}u_{\e})$ and $(\overline{A}_{\e}u_{\e})$ represent the same element in $\mathcal{G}(\overline{\Om})$.  Furthermore, given a representative $(A_{\e})$
of $\mathcal{A}_0$, we also have that $[A_{\e}u_{\e}] = [A_{\e}\overline{u}_{\e}]$ for any two representatives of $u\in \mathcal{G}(\overline{\Om})$.   To see this, we first observe that for each $\e$, every term in $A_{\e}u_{\e}$ is linear except
for the $(u_{\e})^{n_i}$ terms.  So to verify the previous statement it suffices to show that for each
$n_i \in \mathbb{Z}$, $((u_{\e})^{n_i}) = ((\overline{u}_{\e})^{n_i})+(\overline{\eta}_{\e})$, where $(\overline{\eta}_{\e}) \in \mathcal{N}(\overline{\Om})$.  Given that $[(u_{\e})] = [(\overline{u}_{\e})]$ in 
$\mathcal{G}(\overline{\Om})$, we have $(\overline{u}_{\e}) = (u_{\e})+(\eta_{\e})$ for $(\eta_{\e}) \in \mathcal{N}(\overline{\Om})$.  For fixed $\e$, $n_i \in \mathbb{Z}^+$, 
$$
(\overline{u}_{\e})^{n_i} = (u_{\e}+\eta_{\e})^{n_i} = \sum_{j=0}^{n_i} \binom{n_i}{j} (u_{\e})^j(\eta_{\e})^{n_i-j} = (u_{\e})^{n_i} + \overline{\eta}_{\e},
$$
where $\overline{\eta}_{\e}$ consists of the summands that each contain some nonzero power of $\eta_{\e}$.  Clearly the net $(\overline{\eta}_{\e}) \in \mathcal{N}(\overline{\Om})$.  
If $n_i \in \mathbb{Z}^-$, then for a fixed $\e$,
$$
(\overline{u}_{\e})^{n_i} =\frac{1}{(u_{\e}+\eta_{\e})^{|n_i|}} = \frac{1}{\sum_{j=0}^{|n_i|} \binom{|n_i|}{j} (u_{\e})^j(\eta_{\e})^{|n_i|-j} }= \frac{1}{(u_{\e})^{|n_i|} + \overline{\eta}_{\e}}.
$$
By looking at the difference
$$
(u_{\e})^{n_i}- \frac{1}{(u_{\e})^{|n_i|} + \overline{\eta}_{\e}} = \frac{\overline{\eta}_{\e}}{((u_{\e})^{|n_i|})((u_{\e})^{|n_i|}+\overline{\eta}_{\e})} = \hat{\eta}_{\e},
$$
we see that the net $((u_{\e})^{n_i}) = ((\overline{u}_{\e})^{n_i}) + (\hat{\eta}_{\e})$, where $(\hat{\eta}_{\e}) \in \mathcal{N}(\overline{\Om})$.  Therefore for any $u\in \mathcal{G}(\overline{\Om})$
possessing an inverse, and any $A \in \mathcal{A}_0$, the expression $Au = [A_{\e}u_{\e}] \in \mathcal{G}(\overline{\Om})$ is well-defined.
\end{proof}

%%%%%%%%%%%%%%%%%%%%%%%%%%%%%%%%%%%%%%%
\subsection{The Dirichlet Problem in $\mathcal{G}(\overline{\Omega})$}
\label{Dirichlet}

Using the above definition of $\mathcal{A}$, we can now define our semilinear Dirichlet problem
on $\mathcal{G}(\ol{\Om})$.
Let $u,\rho \in \mathcal{G}(\overline{\Omega})$ where $\Om \subset \mathbb{R}^n$ is 
open, bounded and of $C^{\infty}$-class.
% \mnote{check on need for extension operator -cm.}  
Then let $E$ be a total extension operator of $\Om$ such
that for $f \in C^{\infty}(\ol{\Om})$, $Ef \in C^{\infty}(\mathbb{R}^n)$ and $Ef|_{\ol{\Om}} =f$.  See \cite{RA75} for details.
Using $E$ we may 
may define $u|_{\partial\Omega}=\rho|_{\partial\Omega}$ for elements $u, \rho \in \mathcal{G}(\ol{\Om})$
if there are representatives $(u_{\epsilon})$ and $(\rho_{\epsilon})$ such that 
$$
u_{\epsilon}|_{\partial\Omega}=\rho_{\epsilon}|_{\partial\Omega}+n_{\epsilon}|_{\partial\Omega},
$$ 
where $n_{\epsilon}$ is a net of $C^{\infty}$ functions defined in a neighborhood of $\partial\Omega$
such that 
\begin{align}
\label{boundary}
\sup_{x\in\partial\Omega}|n_{\epsilon}(x)|=\textit{o}(\epsilon^a)\hspace{3mm} \forall a\in \mathbb{R}.
\end{align}
This will ensure that $u|_{\partial\Omega}=\rho|_{\partial\Omega}$ does not depend on representatives~\cite{MP06}. 
With this definition of boundary equivalence, for a given operator $A \in \mathcal{A}_0$, the Dirichlet problem
\begin{align}\label{eq1:1nov11}
Au &= 0 \hspace{3mm} \text{in $\Om$},\\
u  &= \rho \quad \text{on $\partial\Om$}  \nonumber
\end{align}
is well-defined in $\mathcal{G}(\overline{\Om})$.  Now we state the conditions
under which the above problem can be solved in $\mathcal{G}(\ol{\Om})$. 

%%%%%%%%%%%%%%%%%%%%%%%%%%%%%%%%%%%%%%%%%%%%%%%%%%%%%%%%%%%%%%%%%%%%%%%%%%%%%%
\section{Overview of the Main Results}
\label{overview}

We begin this section by stating the main existence result for the Dirichlet problem \eqref{eq1:1nov11}.
Let $A\in\mathcal{A}_0$ be an operator on $\mathcal{G}(\overline{\Omega})$ defined by \eqref{eq1june26}.
Also assume that the coefficients of $A$ have representatives
$(a^{ij}_{\e}), (b^i_{\e}) \in \mathcal{E}_M(\overline{\Om})$ that satisfy the following properties for $\e \in (0,1)$:
\begin{align}
\label{eq1june27}
&a^{ij}_{\e} = a^{ji}_{\e}, \hspace{5mm} a^{ij}_{\e}\xi_i\xi_j \ge \lambda_\e|\xi|^2 \ge C_1\e^a|\xi|^2, \\
&|a^{ij}_{\e}|_{k+1,\alpha;\Om}, \hspace{2mm} |b^i_{\e}|_{k,\alpha;\Om} \le \Lambda_{k,\e} \le C_2(k)\e^{b(k)}, \quad \forall k \in \mathbb{N} \nonumber \\
&b^1_{\e} \le -C_3\e^c, \hspace{3mm} \hspace{3mm} \{n_i:n_i< 0\} \ne \emptyset, \hspace{3mm} n_1=\min\{n_i : n_i < 0\} \nonumber \\
&b^K_{\e} \ge C_4\e^d, \hspace{3mm} \{n_i: n_i > 0 \} \ne \emptyset ,\hspace{3mm} n_K=\max\{n_i : n_i > 0\}, \nonumber
\end{align}
where $C_1, C_2, C_3$ and $C_4$ are positive constants independent of $\e$ and the 
constants \\
${a,b,c, d \in \mathbb{R}}$ are also independent of $\e$. The notation $C_2(k)$ and
$b(k)$ is meant to indicate that these constants may depend on $k$.   
Then the following Dirichlet problem has a solution in $\mathcal{G}(\overline{\Om})$:
\begin{align}
\label{eq3june27}
Au =[A_{\e}&u_{\e}] = 0 \hspace{2mm}\text{in  }\Om,\\
u &=\rho  \quad \text{on $\partial\Om$}. \nonumber
\end{align}
We summarize this result in the following theorem, which will be the focus of the remainder of the paper:
\begin{theorem}
\label{thm1june27}
%\mnote{Changed statement of theorem slightly.  Removed assumption that some $b^i$ non constant.}
%\mnote{This is a good change. -mjh}
Suppose that $A:\mathcal{G}(\overline{\Om}) \to \mathcal{G}(\overline{\Om})$ is in $\mathcal{A}_0$ and that the conditions of \eqref{eq1june27} hold.  Assume that $\rho \in \mathcal{G}(\overline{\Om})$
has a representative $(\rho_{\e})$ such that for $\e < 1$, $\rho_{\e} \ge C\e^a$  for some $C>0$ and $a\in \mathbb{R}$.  Then there exists a solution to the Dirichlet problem \eqref{eq3june27} in $\mathcal{G}(\overline{\Om})$. 
\end{theorem}
\begin{proof}
The proof will be given in Section~\ref{results}.
\end{proof}
\begin{remark}\label{rem1:11nov11}
We can actually weaken the assumptions in \eqref{eq1june27} so that the conditions on the representatives  $(a^{ij}_{\e}), (b^1_{\e}),(b^K_{\e}), (\rho_{\e})$
only have to hold for all $\e \in (0,\e_{0})$ for some $\e_0 \in (0,1)$.  Suppose that this is the case, and that using these conditions
we are able to show that for all $\e \in (0,\e_0)$, there exists $u_{\e}$ that solves 
\begin{align}\label{eq1:11nov11}
A_{\e}u_{\e} &= 0 \quad \text{ in $\Om$},\\
u_{\e} &= \rho_{\e} \quad \text{on $\partial\Om$}. \nonumber
\end{align}
If $u_{\e}$ satisfies the additional property that for all $k \in \mathbb{N}$, there exists some $\e_0' \in (0,\e_0)$, $C>0$, and $a \in \mathbb{R}$ such that for all $\e \in (0,\e_0')$, 
$|u_{\e}|_{k,\alpha} \le C\e^a$, then we can form a solution $(v_{\e}) \in \mathcal{E}_M(\ol{\Om})$ to \eqref{eq3june27} by defining $v_{\e} = u_{\e}$ for $\e \in (0, \e_0)$ and 
$v_{\e} = u_{\e_0}$ for $\e \in [\e_0,1]$.  The solution theory that we develop to prove Theorem~\ref{thm1june27} with 
the stronger conditions \eqref{eq1june27} 
will also imply the existence of the partial net $(u_{\e})$ of solutions to \eqref{eq1:11nov11} in the event that the constraints outlined in \eqref{eq1june27} only hold for $\e \in (0,\e_0) \subset (0,1)$.  We will
require this fact when we consider how to embed and solve \eqref{problem} in $\mathcal{G}(\ol{\Om})$ later on in Section~\ref{embed}.
\end{remark}
We begin assembling the tools we will need to prove Theorem~\ref{thm1june27}.  The first tool we need
is a method capable of solving a large class of semilinear problems.  The method of sub- and super-solutions
meets this need, and we discuss this process of solving elliptic, semilinear problems in the following section.

%%%%%%%%%%%%%%%%%%%%%%%%%%%%%%%%%%%%%%%
\subsection{The Method of Sub- and Super-Solutions}
\label{subsolution}

In Theorem~\ref{thm2june27} below, we state a fixed-point result that will be essential in proving Theorem~\ref{thm1june27}.  This fixed-point result is known
as the method of sub- and super-solutions due to the fact that for a given operator $A$, the method relies on finding a sub-solution $u_-$ and super-solution $u_+$ such that
$u_- < u_+$.  A large part of this paper is devoted to finding a net of positive sub- and super-solutions for~\eqref{eq3june27} and establishing growth conditions for them.
In the proof below, let 
\begin{align}
\label{eq1june29}
Lu = -D_i(a^{ij}D_{j}u) + cu,
\end{align}
be an elliptic operator where
$$
a^{ij} = a^{ji}, \quad a^{ij}\xi_i\xi_j \ge \lambda|\xi|^2 \quad \text{and} \quad
a^{ij}, c \in C^{\infty}(\overline{\Om}).
$$
We now state and prove the sub- and super-solution fixed-point result for these assumptions.
\begin{theorem}
\label{thm2june27}
Suppose $\Om \subset \mathbb{R}^n$ is a $C^{\infty}$ domain and assume
$f:\overline{\Om} \times \mathbb{R}^+\to \mathbb{R}$ 
is in
$C^{\infty}(\overline{\Om}\times\mathbb{R}^+)$ and $\rho \in C^{\infty}(\overline{\Om})$.  Let $L$ be of the
form \eqref{eq1june29}.  Suppose that there exist functions
$u_-:\overline{\Om} \to\mathbb{R}$ and $u_+:\overline{\Om}\to \mathbb{R}$ 
such that the following hold:
\begin{enumerate}
\item $u_-,u_+ \in C^{\infty}(\overline{\Om})$,
\item $0<u_-(x) \le u_+(x) \hspace{3mm}\forall x\in \overline{\Om}$,
\item $ Lu_- \le f(x,u_-)$,
\item $ Lu_+ \ge f(x,u_+)$,
\item $ u_- \le \rho \hspace{2mm}\text{on} \hspace{2mm} \partial\Om$,
\item $ u_+ \ge \rho \hspace{2mm}\text{on} \hspace{2mm} \partial\Om$.
\end{enumerate}
Then there exists a solution $u$ to 
\begin{align}
\label{eq3june30}
Lu &= f(x,u) \hspace{3mm} \text{on $\Om$},\\
u &= \rho \quad \text{on $\partial\Om$}, \nonumber
\end{align}
such that
\begin{itemize}
\item[(i)] $u\in C^{\infty}(\overline{\Om})$,
\item[(ii)] $u_-(x)\le u(x) \le u_+(x). $
\end{itemize}
\end{theorem}  
\begin{proof}
The general approach of the proof will be to construct a monotone sequence $\{u_n\}$ that is point-wise bounded above and below by our super- and sub-solutions, $u_+$ and $u_-$.  
We will then apply elliptic regularity estimates and the Arzela-Ascoli Theorem to conclude that the sequence $\{u_n\}$ has a $C^{\infty}(\overline{\Om})$ limit $u$ that is a solution to 
\begin{align}
Lu &= f(x,u) \hspace{3mm}\text{on $\Om$},\\
u &= \rho \quad \text{on $\partial\Om$}. \nonumber
\end{align}
 
Given that $u_-(x), u_+(x) \in C^{\infty}(\overline{\Om})$, the interval $[\min u_-(x), \max_+ u_+(x)] \subset \mathbb{R}^+$ is well-defined.  We then restrict the domain of the function $f$ to the compact set
$K = \overline{\Om}\times [\min u_-(x), \max_+ u_+(x)]$.  Given that
$f \in C^{\infty}(\overline{\Om}\times \mathbb{R}^+)$, it is clearly in $C^{\infty}(\overline{\Om}\times [\min u_-(x), \max_+ u_+(x)])$ and so the function $|\frac{\partial f(x,t)}{\partial t}|$ is continuous and attains a maximum on $ K$.  Denoting this maximum value by
$m$, let $M= \max\{m, -\inf_{x\in \ol{\Om}} c(x)\}$.  Then consider the
operator 
$$
Au =  Lu + Mu,
$$
and the function
$$
F(x,t) = Mu+f(x,t).
$$
Note that this choice of $M$ ensures that $F(x,t)$ is an increasing function in $t$ on $K$
and that $A$ is an invertible operator.  
Also, we clearly have the following: 
\begin{align}
A(u) = F(x,u) &\Longleftrightarrow Lu = f(x,u), \\
A(u_-) \le F(x,u_-) &\Longleftrightarrow L(u_-) \le f(x,u_-), \\
A(u_+) \ge F(x,u_+) &\Longleftrightarrow L(u_+) \ge f(x,u_+).
\end{align}
The first step in the proof is to construct the sequence $\{u_n\}$ iteratively.  Let $u_1$ satisfy the equation
\begin{align}
A(u_1) &= F(x,u_-) \hspace{2mm} \text{on $\Om$},\\
u_1&= \rho \quad \text{on $\partial\Om$}. \nonumber
\end{align}
We observe that for $u, v\in H^1_0(\Om)$, the operator $A$ satisfies
$$
C_1\|u\|_{H^1(\Om)}^2 \le \left\langle Au,u \right\rangle, \hspace{5mm} \text{ and} \hspace{5 mm} \left\langle Au,v\right\rangle \le \|u\|_{H^1(\Om)}^2\|v\|_{H^1(\Om)}^2,
$$
where 
$$
\left\langle u,v\right\rangle = \int_{\Om} uv dx, \hspace{5mm}\text{and} \hspace{5 mm} \left\langle Lu,v\right\rangle = \int_{\Om} (a^{ij} D_ju D_iv + cuv) dx.
$$
Therefore the Lax-Milgram theorem implies
that there exists a weak solution ${u_1 \in H^1(\Om)}$ satisfying ${u_1-\rho \in H^1_0(\Om)}$.
Given our assumptions on $F(x,t)$ and $\rho$,  
${F(x,u_+) \in H^{m}(\Om)}$ and ${\rho \in H^m(\Om)}$ for all $m \in \mathbb{N}$.
Therefore, by standard elliptic regularity arguments, $u_1 \in H^m(\Om)$ for all $m\in \mathbb{N}$.
This, the assumption that $\Om$ is of $C^{\infty}$-class and the assumption that $a^{ij}, c, \rho \in C^{\infty}(\ol{\Om})$ imply that $u_1\in C^{\infty}(\overline{\Om})$
and $u_1 = \rho \hspace{2mm} \text{on $\partial\Om$}$.  Therefore, we may iteratively define the sequence $\{u_j\} \subset C^{\infty}(\overline{\Om})$ where
\begin{align}
\label{eq3june29}
A(u_{j}) &= F(x,u_{j-1}) \hspace{3mm} \text{on $\Om$},\\
       u_{j} &= \rho \quad \text{on $\partial\Om$}. \nonumber
\end{align}

The next step is to verify that the sequence $\{u_j\}$ is a monotonic increasing sequence satisfying $u_- \le u_1 \le \cdots \le u_{j-1} \le u_j \le \cdots \le u_+$.
We prove this by induction.  First we observe that
\begin{align}
A(u_- -u_1)  \le F(x,u_-)-F(x,u_-) &= 0 \hspace{3mm}\text{on $\Om$}, \\
  ( u_-- u_1)|_{\partial\Om} &\le 0. \nonumber
\end{align}
Therefore, by the weak maximum principle, $u_-\le u_1$ on $\overline{\Om}$.  Now suppose that $u_{j-1} \le u_{j}$.  Then
\begin{align}
A(u_j-u_{j+1}) = F(x,u_{j-1})-F(x,u_j) &\le 0 \hspace{3mm}\text{ on $\Om$}, \\
  ( u_j-u_{j+1})|_{\partial\Om} &= 0 . \nonumber
\end{align}
given that $F(x,t)$ is an increasing function in the variable $t$ and $u_{j-1} \le u_j$.  The weak maximum 
principle again implies that $u_j \le u_{j+1}$, so by induction
we have that $\{u_j\}$ is monotonic increasing sequence that is point-wise bounded below by $u_-(x)$.  
Now we show that our increasing sequence is point-wise bounded above by $u_+(x)$ by proceeding in a 
similar manner.  Given that $u_- \le u_+$ and $u_+$ is a super-solution, we have that
\begin{align}
A(u_1-u_+) \le F(x,u_-)-F(x,u_+) &\le 0 \hspace{3mm} \text{ on $\Om$},\\
(u_1-u_+)|_{\partial\Om} &\le 0 . \nonumber
\end{align}
The weak maximum principle implies that $u_1\le u_+$.  Now assume that $u_j \le u_+$.  Then 
\begin{align}
A(u_{j+1}-u_+) \le F(x,u_j)-F(x,u_+) &\le 0 \hspace{3mm} \text{on $\Om$},\\
(u_{j+1} - u_+)|_{\partial\Om} &\le 0, \nonumber
\end{align}
given that $F(x,t)$ is an increasing function and $u_j \le u_+$.  So by induction the sequence 
$\{u_j\}$ is a monotonic increasing sequence that is point-wise bounded above by $u_+(x)$ and 
point-wise bounded below by $u_-(x)$.  

Up to this point, we have constructed a monotonic increasing sequence ${\{u_j\} \subset C^{\infty}(\overline{\Om})}$ such that for each $j$,
$u_j$ satisfies the Dirichlet problem \eqref{eq3june29} and 
is point-wise bounded below by $u_-$ and above by $u_+$.  The next step will be to apply the Arzela-Ascoli theorem and a bootstrapping
argument to conclude that this sequence converges to $u\in C^{\infty}(\overline{\Om})$.  We first show that it converges to $u\in C(\overline{\Om})$
by an application of the Arzela-Ascoli Theorem.   Clearly the family of functions $\{u_j\}$ is point-wise bounded, so it is only necessary to establish
the equicontinuity of the sequence.  Given that each function $u_j$ solves the problem \eqref{eq3june29},
by standard $L^p$ elliptic regularity estimates
(cf.~\cite{GiTr77}) 
% \mnote{ any need to state a theorem about this? -cm}
we have that
$$
\|u_j\|_{W^{2,p}} \le C(\|u_j\|_{L^p}+\|F(x,u_{j-1})\|_{L^p}).
$$
The regularity of $F(x,t)$ and the sequence $\{u_j\}$ along with the above estimate and the compactness of $\overline{\Om}\times [\inf u_-, \sup u_+]$ imply that
there exists a constant $N$ such that $  \|F(x,u_{j-1})\|_{L^p} \le N$ for all $j$.  Therefore, if $p> 3$, the above  bound
and the fact that $u_- \le u_j \le u_+$ imply that for each $j\in \mathbb{N}$,
 $$
 |u_j|_{1,\alpha;\Om} \le C\|u\|_{W^{2,p}} \le \infty,
 $$
 where $\alpha = 1-\frac3p$.  This implies that the sequence $\{u_j\}$ is equicontinuous.  The Arzela-Ascoli Theorem
 then implies that there exists a $u\in C(\overline{\Om})$ and a subsequence $\{u_{j_k}\}$ such that $u_{j_k} \to u$ uniformly. 
 Furthermore, due to the fact that the sequence $\{u_j\}$ is monotonic increasing, we actually have 
 that $u_j \to u$ uniformly on $\overline{\Om}$.   
Once we have that $u_{j} \to u $ in $C(\overline{\Om})$, we apply $L^p$ regularity theory
again to conclude that
\begin{align}\label{eq1june30}
|u_{j} - u_{k}|_{1,\alpha;\Om} \le& C\|u_j-u_k\|_{W^{2,p}}\\
 \le& C'(\|u_{j}-u_{k}\|_{L^p}+\|F(x,u_{j-1})-F(x,u_{k-1})\|_{L^p}).\nonumber
\end{align}      
Note that the above estimate follows from the fact that $u_{j_k+1}-u_{j_l+1}$ satisfies 
\begin{align}
A(u_{j}-u_{k}) &= F(x,u_{j-1})-F(x,u_{k-1}) \hspace{3mm} \text{on $\Om$}, \\
(u_{j} - u_{k})|_{\partial\Om} &= 0. \nonumber
\end{align}
Given that $u_j \to u$ in $C(\overline{\Om})$, \eqref{eq1june30} 
 implies that the sequence $\{u_{j}\}$ is a Cauchy sequence in $C^{1}(\overline{\Om})$.  The completeness of $C^{1}(\overline{\Om})$ 
 then implies that this subsequence has a limit  $v \in C^{1}(\overline{\Om})$, and given that $u_{j} \to u$ in $C(\overline{\Om})$, it follows that $u = v$.  
Similarly, by repeating the above argument and using higher order $L^p$ estimates we have that
 \begin{align}\label{eq2june30}
 |u_{j} - u_{k}|_{2,\alpha;\Om} \le &C(\|u_j-u_k\|_{W^{3,p}} ) \\
  \le & C'(\|u_j-u_k\|_{W^{1,p}}+\|F(x,u_{j-1})-F(x,u_{k-1})\|_{W^{1,p}})\nonumber,
 \end{align}
 where $u_{j} \to u$ in $C^{1}(\overline{\Om})$ as $k \to \infty$.  Again, \eqref{eq2june30}, the regularity of $F$ and the fact
 that $u_j \to u$ in $C^1(\ol{\Om})$ 
 imply that the sequence $\{u_{j} \}$ is Cauchy in $C^{2}(\overline{\Om})$.  A simple induction
 argument then shows that $u \in C^{\infty}(\overline{\Om})$.
 
 The final step of the proof is to show that $u$ is an actual solution to the problem~\eqref{eq3june30}.
 It suffices to show that $u$ is a weak solution to the above problem.  It is clear that $u = \rho \hspace{2mm} \text{on $\partial\Om$}$,
 so we only need to show that $u$ satisfies \eqref{eq3june30} on $\Om$.   
 Fix $v \in H^1_0(\Om)$.  Then based on the definition
 of the sequence $\{u_j\}$, we have
 $$
 \int_{\Om}( a^{ij}D_j u_j D_i v + Mu_{j}v )dx = \int_{\Om} (f(x,u_{j-1})+ Mu_{j-1})v dx.
 $$
 As $u_j \to u$ uniformly in $C(\overline{\Om})$, we let let $j\to \infty$ to conclude that
 $$
  \int_{\Om}( a^{ij}D_j u D_i v + Muv )dx = \int_{\Om} (f(x,u)+ Mu)v dx.
  $$
  Upon canceling the term involving $M$ from both sides, we find that $u$ is a weak solution.
 \end{proof}

%%%%%%%%%%%%%%%%%%%%%%%%%%%%%%%%%%%%%%%
\subsection{Outline of the Proof of Theorem~\ref{thm1june27}}
\label{over}

Now that the sub- and super-solution fixed-point theorem is in place, we give an outline for how to prove Theorem~\ref{thm1june27}.  
\begin{itemize}
\label{steps}
\item[Step 1:]  { \it Formulation of the problem.}  We phrase~\eqref{eq3june27} in a way that allows us to solve a net of semilinear
elliptic problems.  We assume that the coefficients of $A$ and boundary data $\rho$ have representatives $(a^{ij}_{\e}), (b^i_{\e}),$ and $(\rho_{\e})$
in $\mathcal{E}_M(\overline{\Om})$ satisfying the assumptions \eqref{eq1june27}.
Then for this particular choice of representatives, we solve the family of problems:
\begin{align}\label{11feb15eq1}
A_{\epsilon}u_{\epsilon} = -\sum_{i,j=1}^N D_i( a_{\epsilon}^{ij}D_j u_{\epsilon}) &+ \sum_i^N b^i_{\epsilon}u_{\epsilon}^{n_i}=0 \quad \text{  in  $\Omega$},\\
u_{\epsilon} &= \rho_{\epsilon} \quad \text{on $\partial\Om$}. \nonumber 
\end{align}
Then we must ensure that the net of
solutions $(u_{\epsilon})\in \mathcal{E}_M(\overline{\Omega})$ and ensure that~\eqref{11feb15eq1} is satisfied for other representatives of $A,\rho,u$.

\item[Step 2:] {\it Determine $L^{\infty}$-estimates and a net of generalized constant sub-solutions and super-solutions}.  We determine constant, {\em a priori} $L^{\infty}$
bounds such that for a positive net of solutions $(u_{\e})$ of the semilinear problem \eqref{11feb15eq1}, there exist constants $a_1,a_2 \in \mathbb{R}$,
$C_1, C_2>0$ independent of $\e\in (0,1)$ such that 
$$C_1\e^{a_1}<\alpha_{\e} \le u_{\e} \le \beta_{\e} <C_2\e^{a_2}.$$ 
These estimates are constructed in such a way that for each $\e$,
the pair $\alpha_{\e}, \beta_{\e}$ are sub- and super-solutions for \eqref{11feb15eq1}.    
\item[Step 3:]{\it Apply fixed-point theorem to solve each semilinear problem in \eqref{11feb15eq1}}. Using the sub- and super-solutions $\alpha_{\e}, \beta_{\e}$, 
 we apply Theorem~\ref{thm2june27} to obtain a net of solutions $(u_{\e}) \in C^{\infty}(\overline{\Om})$.  

\item[Step 4:]{\it Verify that the net of solutions $(u_{\e}) \in \mathcal{E}_M(\overline{\Om})$}.  Here we show that the net of solutions
satisfies the necessary growth conditions in $\e$ using the growth conditions on the sub- and super- solutions  and Theorem~\ref{thm1june30}.

\item[Step 5:]{\it Verify that the solution is well-defined}.   Once we've determined that the net of solutions $(u_{\e}) \in \mathcal{E}_M(\overline{\Om})$,
we conclude that $[(u_{\e})] \in \mathcal{G}(\overline{\Om})$ is a solution to the Dirichlet problem \eqref{eq3june27} by showing that the solution
is independent of the representatives chosen.   Note that most of the work for this  step was done in Proposition~\ref{prop1july1}.

\end{itemize}

We shall carry out the above steps in our proof of Theorem~\ref{thm1june27} in Section~\ref{results}.  We still need
to determine a net of sub- and super- solutions for \eqref{eq1june27}, which  we do in Section~\ref{bounds1}.  But
before we move on to this and the other steps in the above outline, we briefly return to the motivating problem \eqref{problem}
by discussing how to embed a problem with distributional data into $\mathcal{G}(\ol{\Om})$. 

%%%%%%%%%%%%%%%%%%%%%%%%%%%%%%%%%%%%%%%
\subsection{Embedding a Semilinear Elliptic PDE with Distributional Data into $\mathcal{G}(\ol{\Om})$.}
\label{embed2}
%\mnote{Changed this section to coincide with embedding section.}
%\mnote{Looks OK. -mjh}
Now that we have defined what it means to solve a differential equation in $\mathcal{G}(\ol{\Om})$, we are ready
to return to the problem discussed at the beginning of the paper.  We are interested in solving
an elliptic, semilinear Dirichlet problem of the form
\begin{align}\label{eq1:27oct11}
-\sum_{i,j=1}^ND_i(a^{ij}D_ju) &+ \sum^K_{i=1}b^i u^{n_i} = 0 \quad \text{ in $\Om$},\\
u &= \rho \quad \text{on $\partial\Om$} \nonumber,
\end{align}
where $a^{ij}, b^i$ and $\rho$ are potentially distributional and $n_i \in \mathbb{Z}$ for each $i$.
If we can formulate this problem as a family of equations similar to \eqref{11feb15eq1}, 
then it can readily be solved in $\mathcal{G}(\ol{\Om})$ by Theorem~\ref{thm1june27}.  
The key to formulating our problem with singular data as a net of problems is Theorem~\ref{thm1:5feb13}. 

Suppose that $\Om' \subset\subset \Om$ and $\Om$ is of $C^{\infty}$-class.
For this choice of $\Om'$, we can construct
an extension operator $E$ as in Theorem~\ref{thm1:4feb13} and then use Theorem~\ref{thm1:5feb13} to define an embedding of $\cF'(\Om)$ into
$\cG(\ol{\Om})$, where we defined $\cF'(\Om)\subset \cD'(\Om)$ in section~\ref{embed}.  If we are given a problem of the form \eqref{eq1:27oct11} 
with data $a^{ij}, b^i, \rho$ in $\mathcal{F}'(\Om)$, 
then we may use Theorem~\ref{thm1:5feb13} to embed the coefficients $a^{ij}, b^i$ and $\rho$ into $\mathcal{G}(\ol{\Om})$.  
We will denote a representative of the image of each these terms
in $\mathcal{G}(\ol{\Om})$ by $(a^{ij}_{\e}), (b^i_{\e})$ and $(\rho_{\e})$.  Then for a choice of representatives,
we obtain a net of problems of the form \eqref{11feb15eq1}.  

In order to solve this net of problems
using Theorem~\ref{thm1june27}, we need there to exist a choice of representatives $(a^{ij}_{\e}), (b^i_{\e})$ and $(\rho_{\e})$
that satisfy the conditions specified in \eqref{eq1june27}.  While these conditions might seem exacting, this
solution framework still admits a wide range of interesting problems.  This is evident when one considers
the following proposition:

\begin{proposition}\label{eq1:24feb12}
%\mnote{Changed this prop. to coincide with changes made in Section~\ref{embed}.}
%\mnote{Looks OK. -mjh}
Let $\Om'\subset\subset \Om$, where $\Om$ is bounded and of $C^{\infty}$-class, and define $\cF'(\Om)$ as in section~\ref{embed}.
Let $n_i \in \mathbb{Z}$ be a collection of integers for $1 \le i \le K$ and assume that there exist $1 \le i,j \le K$ such that 
$n_i <0$ and $n_j >0$.  Then assume that
$$n_1 = \min\{n_i:~ n_i < 0 \}, \quad \text{and} \quad n_K = \max\{n_i:~ n_i > 0\}.$$ 
Suppose that  $a^{ij}, b^1, b^K, \rho \in C(\ol{\Om})$ and $b^2, \cdots, b^{K-1} \in \mathcal{F}'(\Om)$.
Additionally assume that $a^{ij}$ satisfies the symmetric, ellipticity condition and $\rho>0$, $b_1<0$ and
$b_K>0$ in $\Om$.  Then the problem
\begin{align}\label{eq1:28oct11}
-\sum_{i,j=1}^ND_i(a^{ij}D_ju) &+ \sum^K_{i=1}b^i u^{n_i} = 0 \quad \text{ in $\Om$},\\
u &= \rho \quad \text{on $\partial\Om$} \nonumber,
\end{align}
admits a solution in $\mathcal{G}(\ol{\Om})$.
\end{proposition}    
\begin{proof}
This follows from Proposition~\ref{thm1:5feb13}, Theorem~\ref{thm1june27}, Remark~\ref{rem1:11nov11} and the fact that \\$(a^{ij}\ast \phi_{\e})$,
 $(b^1\ast \phi_{\e})$, $(b^K\ast \phi_{\e})$ and $(\rho \ast \phi_{\e})$
 converge uniformly to $a^{ij}, b^1, b^K$ and $\rho$ in $\ol{\Om}$.  For $\e$ sufficiently small,
the corresponding problem \eqref{11feb15eq1} in 
$\mathcal{G}(\ol{\Om})$ will satisfy the conditions specified in \eqref{eq1june27}.  Therefore, Theorem~\ref{thm1june27} and Remark~\ref{rem1:11nov11}
imply the result. 
\end{proof}
With the issue of solving \eqref{eq1:27oct11} at least partially resolved, we return to the task of proving Theorem~\ref{thm1june27}.
We begin by establishing some {\em a priori} $L^{\infty}$-bounds for a solution to our semilinear problem \eqref{eq1:28oct11} if the given data is smooth.

%%%%%%%%%%%%%%%%%%%%%%%%%%%%%%%%%%%%%%%%%%%%%%%%%%%%%%%%%%%%%%%%%%%%%%%%%%%%%%
\section{Sub- and Super-Solution Construction and Estimates}
\label{bounds1}

Given an operator $A \in \mathcal{A}_0$ with coefficients satisfying \eqref{eq1june27}, our solution strategy 
for the Dirichlet problem \eqref{eq3june27} is to solve the family of problems \eqref{11feb15eq1} and then
establish the necessary growth estimates.  In order for this to be a viable strategy, we first need to show
that \eqref{11feb15eq1} has a solution for each $\e \in (0,1)$.
Given that $n_i <0$ for some $1 \le i \le K$, for each $\e$, we must restrict
the operator
 $$A_{\e}u_{\e} = -\sum_{i,j=1}^N D_i(a_{\e}^{ij}D_j u_{\e})+ \sum_{i=1}^K b^i u_{\e}^{n_i},$$ 
to a subset of functions in $C^{\infty}(\ol{\Om})$
to guarantee that $A_{\e}$ is well-defined.  In particular, for each $\e$ we consider functions $u_{\e}\in C^{\infty}(\ol{\Om})$
such that $0 < \alpha_{\e}\le  u_{\e} \le \beta_{\e} < \infty$ for some choice of $\alpha_{\e}$ and $\beta_{\e}$.
The first part of this section is dedicated to making judicious choices of $\alpha_{\e}$ and $\beta_{\e}$ for each $\e$
such that a solution $u_{\e}$ to \eqref{11feb15eq1} exists that satisfies $\alpha_{\e} \le u_{\e} \le \beta_{\e}$.

Once a net of solutions $(u_{\e})$ is determined, it is necessary to show that if $(u_{\e}) \in \mathcal{E}_M(\ol{\Om})$,
then an operator $A \in \mathcal{A}_0$ whose
coefficients satisfy \eqref{eq1june27} is well-defined for $(u_{\e})$.  Recall that $A$ is only a well defined operator
for elements $u \in \mathcal{G}(\ol{\Om})$ satisfying $u_{\e} \ge C\e^a$
for $\e \in (0,\e_0) \subset (0,1)$, $a\in \mathbb{R}$ and some constant $C$ independent of $\e$. 
This will require us to establish certain
$\e$-growth estimates on $\alpha_{\e}$, which we do later in this section. 

%%%%%%%%%%%%%%%%%%%%%%%%%%%%%%%%%%%%%%%
\subsection{$L^{\infty}$ Bounds for the Semilinear Problem}
\label{bounds}
We begin by determining the net of {\em a priori} bounds $\alpha_{\e}$ and $\beta_{\e}$ described above.
For now we disregard the $\e$ notation.   
In the following proposition we
 determine {\em a priori} estimates for a weak solution $u\in H^{1}(\Om)$ to a problem of the form
\begin{align}
\label{eq1june20}
-\sum_{i,j}^ND_i (a^{ij} D_j u) &+ \sum_{i=1}^K b^iu^{n_i}=0 \text{  in  $\Om$}, \\
u &= \rho  \quad \text{on $\partial\Om$} ,\nonumber
\end{align}
with certain conditions imposed on the coefficients and exponents.  In particular,
in the following proposition we assume that $\Om \subset \mathbb{R}^n$ is connected, 
bounded, and of $C^{\infty}$-class, and $a^{ij}, b^i, \rho \in C(\overline{\Om})$ with
$\rho >0$ in $\ol{\Om}$.   

\begin{proposition}
\label{prop1july20}
%\mnote{Changed the statement and proof of this proposition slightly.  Included explicit $L^{\infty}$ bounds in statement of prop.}
%\mnote{This is an improvement. -mjh}
Suppose that the semilinear operator in \eqref{eq1june20} has the property that
$n_i>0$ for some 
$1 \le i \le K$.  Let $n_K$ be the largest positive exponent and 
suppose that $b_K(x)>0$ in $\overline{\Om}$. 
Additionally, assume that one of the following two cases holds:
\begin{align}
\text{{(1)}}
  &\text{ $n_i < 0$ for some $1 \le i < K$ and if $n_1= \min\{n_i: n_i < 0\}$},
\label{eq1:10feb13}
\\
  & \qquad \text{then $b^1(x) < 0$ in $\overline{\Om}$.}
\nonumber
\\
\text{{(2)}}
  &\text{ $n_K$ is odd and $0<n_i~$  for all $1 \le i\le K$. }
\label{eq2:10feb13}
\end{align} 
If case \eqref{eq1:10feb13}  holds, define
\begin{align}\label{eq3:10feb13}
&\alpha_1' = \sup_{c\in\mathbb{R}_+} \left\{\sum_{i=1}^K\sup_{x\in\ol{\Omega}}b^i(x)y^{n_i}< 0 \hspace{3mm}\forall y\in (0,c)\right\},\\
&\alpha_1 = \min\{\alpha_1', \inf_{x\in \partial\Om}\rho(x)\}.
\end{align}
If case \eqref{eq2:10feb13} holds, define
\begin{align}
&\alpha_2' = \sup_{c\in\mathbb{R}} \left\{\sum_{i=1}^K\sup_{x\in\ol{\Omega}}b^i(x)y^{n_i}< 0 \hspace{3mm}\forall y\in (-\infty,c)\right\},\\
&\alpha_2 = \min\{\alpha_2', \inf_{x\in \partial\Om}\rho(x)\}.
\end{align}
If case \eqref{eq1:10feb13} or case \eqref{eq2:10feb13} holds, define
\begin{align}\label{eq4:10feb13}
&\beta' = \inf_{c\in\mathbb{R}} \left\{\sum_{i=1}^K\inf_{x\in\ol{\Omega}}b^i(x)y^{n_i}> 0 \hspace{3mm}\forall y\in (c,\infty)\right\},\\
&\beta = \max\{\beta', \sup_{x\in \partial\Om} \rho(x)\}.
\end{align}

\noindent
Under these assumptions and definitions, if case \eqref{eq1:10feb13} holds and $u \in H^1(\Om)$ is a positive
weak solution to Eq. \eqref{eq1june20}, 
then $0< \alpha_1 \le u \le \beta < \infty $.
Otherwise, if case \eqref{eq2:10feb13} holds and $u \in H^1(\Om)$ is a weak solution to Eq. \eqref{eq1june20},
then $ -\infty < \alpha_2 \le u \le \beta < \infty$.  
\end{proposition}

\begin{remark}
%\mnote{Remark new.  Interprets above prop. and addresses fact that \eqref{eq1june20} ill-defined for arbitrary $u \in H^1(\Om)$.} 
%\mnote{Very good idea to point this out explicitly. -mjh}
We observe that Eq. \eqref{eq1june20} does not have a well-defined weak formulation for arbitrary
$u \in H^1(\Om)$.  The way to interpret Proposition~\ref{prop1july20} is that if we seek a positive solution
$u \in H^1(\Om)$ that weakly solves Eq. \eqref{eq1june20} and satisfies condition \eqref{eq1:10feb13}, then
we only need to look for solutions in $H^1(\Om)\cap [\alpha_1,\beta]$, where $[\alpha_1,\beta]$ denotes
the $L^{\infty}(\Om)$ interval of functions $u$ such $\alpha_1\le u \le  \beta~a.e$.  Similarly, we only need to look
for $u \in H^1(\Om)\cap [\alpha_2,\beta]$ if condition \eqref{eq2:10feb13} holds.  
\end{remark}

\begin{remark}
Note that for the purposes of proving Theorem~\ref{thm1june27}, we are primarily concerned with case \eqref{eq1:10feb13}.
This is the case that we will focus on for the remainder of the paper.  However, with a little extra work we could very easily generalize
Theorem~\ref{thm1june27} to allow for $n_i > 0$ for all $1 \le i \le K$ and $n_K> 0$ odd.  Then we could use case \eqref{eq2:10feb13}
to establish the necessary bounds.  
\end{remark}

\begin{proof}
We first note that in all cases $\alpha_1$, $\alpha_2$ and $\beta$ are
well-defined given the conditions on $b^1(x)$ and $b^K(x)$ and the exponents $n_i$ for $1 \le i \le K$.  In particular,
the assumption that $b^1(x) < 0$ in \eqref{eq1:10feb13} ensures that $\alpha_1$ is well-defined and the assumption
that $n_K$ is odd and $n_i >0$ ensures that $\alpha_2$ is well-defined.
 
Based on the definitions of $\alpha_1$, $\alpha_2$ and $\beta$, if $u$ is a solution to \eqref{eq1june20} (we assume $u$ is nonnegative in the case of \eqref{eq1:10feb13}), 
then it is easy to verify that the functions $\overline{\phi}_1 = (u-\beta)^+$ and $\underline{\phi}_1 = (u-\alpha_1)^-$
are in $H^1_0(\Om)$ if \eqref{eq1:10feb13} holds and 
$\underline{\phi}_2 = (u-\alpha_2)^-$ and $\overline{\phi}_2 = (u-\beta)^+$ are in $H_0^1(\Om)$ if \eqref{eq2:10feb13} holds.  
Indeed, we may write $u = u_0+u_D$, where $u_0 \in H^{1}_0(\Om)$ and we have that
\begin{align}
&0 \le \ol{\phi}_1 = (u-\beta)^+ = (u_0+u_D -\beta)^+ \le (u_D - \beta)^+ +u_0^+ \label{eq5:10feb13}\\
&0 \ge \underline{\phi}_1 = (u-\alpha_1)^- = (u_0+u_D - \alpha_1)^- \ge (u_D-\alpha_1)^- + u_0^- . \label{eq6:10feb13}
\end{align}
Taking the trace of Eqs. \eqref{eq5:10feb13} and \eqref{eq6:10feb13} and using the definition of $\alpha_1$ and $\beta$ we find
that $\underline{\phi}_1$ and $\ol{\phi}_1$ are in $H^1_0(\Om)$.  By applying a similar argument
we can conclude that $\underline{\phi}_2 \in H^1_0(\Om)$.

Define the set 
\begin{align*}
\overline{\mathcal{Y}} &=\left\{x\in\overline{\Om}~|~u\ge \beta \right\} 
\end{align*}
if case \eqref{eq1:10feb13} or \eqref{eq2:10feb13} holds.  If case \eqref{eq1:10feb13} holds, let
\begin{align*}
\underline{\mathcal{Y}}_1 &= \{x\in\overline{\Om}~|~0< u \le \alpha_1\},
\end{align*}
and if case \eqref{eq2:10feb13} holds, let 
\begin{align*}
\underline{\mathcal{Y}}_2 &=\{x \in \overline{\Om}~|~u < \alpha_2\}.
\end{align*}
Then if $u \in H^1(\Om)^+$ is a weak solution to~\eqref{eq1june20}, supp($\underline{\phi}_1$) = $\underline{\mathcal{Y}}_1$.
Similarly, if $u \in H^1(\Om)$ is a weak solution to $\eqref{eq1june20}$, then supp($\overline{\phi}_1$) = supp($\ol{\phi}_2$) = $\overline{\mathcal{Y}}$ and supp($\underline{\phi}_2$) = $\underline{\mathcal{Y}}_2$.

We have the following string of inequalities for $\underline{\phi}_1$ if condition \eqref{eq1:10feb13} holds:
\begin{align}
C_2\|\underline{\phi}_1\|^2_{H^1(\Om)} &\le C_1\|\nabla((u-\alpha)^-)\|^2_{L^2(\Om)} \\
&\le \int_{\Om} a^{ij} D_j((u-\alpha)^-)D_j ((u-\alpha)^-) ~dx \nonumber \\
& =  \int_{\Om}a^{ij} D_j(u-\alpha)D_j ((u-\alpha)^-)~ dx \nonumber \\
&= \int_{\underline{\cY}_1}(-\sum_{i=1}^K b^i(x)u^{n_i}) (u-\alpha)~ dx \le 0. \nonumber
\end{align}
We can make a similar argument to show that $\|\underline{\phi}_2\|_{H^1(\Om)} = 0$ if condition \eqref{eq2:10feb13} holds.

We also have the following string of inequalities for $\overline{\phi} = \ol{\phi}_1 = \ol{\phi}_2$
if either condition \eqref{eq1:10feb13} or \eqref{eq2:10feb13} holds:
\begin{align}
C_2\|\ol{\phi}\|^2_{H^1(\Om)} & \le C_1\|\nabla((u-\beta)^+)\|^2_{L^2(\Om)} \\
&\le \int_{\Om} a^{ij} D_j ((u-\beta)^+)D_i ((u-\beta)^+)~ dx \nonumber \\
&= \int_{\Om} a^{ij} D_j (u-\beta)D_i ((u-\beta)^+) ~dx \nonumber \\
&=\int_{\overline{\mathcal{Y}}}(-\sum_{i=1}^K b^i(x)u^{n_i})(u-\beta) ~dx \le 0. \nonumber
\end{align}
The above inequalities imply the result.
\end{proof}

Now that we've established $L^{\infty}$-bounds for solutions to \eqref{eq1june20}, we can apply these bounds for each fixed $\e$ to determine a net of bounds 
for the following net of problems:
 
\begin{align}
\label{eq1july1}
 -\sum_{i,j}^ND_i a_{\e}^{ij} D_j u_{\e} &+ \sum_{i=1}^K b_{\e}^i u_{\e}^{n_i}=0\quad \text{  in  $\Om$}\\
u_{\e} &= \rho_{\e} \quad \text{on $\partial\Om$} , \nonumber 
\end{align}
where $(a^{ij}_{\e}), (b^i_{\e}), (\rho_{\e}) \in \mathcal{E}_M(\ol{\Om})$ satisfy the following for all $\e <1$:
\begin{align}
\label{eq1july2}
 &a^{ij}_{\e}  = a^{ji}_{\e}, \hspace{5mm} a^{ij}_{\e}\xi_i\xi_j \ge \lambda_\e|\xi|^2 \ge C_1\e^{a_1}|\xi|^2  \\
 &|a^{ij}_{\e}|_{k,\alpha;\Om}, \hspace{2mm} |b^i_{\e}|_{k,\alpha;\Om} \le \Lambda_{k,\e} \le C_2(k)\e^{a_2(k)}, \hspace{3mm} \forall k \in \mathbb{N} \nonumber \\
 &b^1_{\e} \le -C_3\e^{a_3}, \hspace{3mm} \{n_i: n_i< 0 \} \ne \emptyset, \hspace{3mm} n_1=\min\{n_i : n_i < 0\}  \nonumber \\
 &b^K_{\e} \ge C_4\e^{a_4}, \hspace{3mm} \{n_i: n_i > 0\} \ne \emptyset, \hspace{3mm} n_K=\max\{n_i : n_i > 0\} \nonumber \\
 &\rho_{\e} \ge C_5\e^{a_5} \nonumber,
\end{align}
and $C_1,\cdots,C_5$ are positive constants that are independent of $\e$ and $a_1,\cdots,a_5 \in \mathbb{R}$ are
independent of $\e$.  Then notation $C_2(k)$ and $a_2(k)$ is meant to indicate that these constants may depend on $k$. 

\begin{proposition}
\label{prop2july20}
Suppose that for each fixed $\e \in (0,1]$, $u_{\e}$ is a positive solution to~\eqref{eq1july1}
with coefficients satisfying~\eqref{eq1july2}.
 Then there exist $L^{\infty}$-bounds $\alpha_{\e}$ and $\beta_{\e}$ such that for each $\e$,
$0< \alpha_{\e} \le u_{\e} \le \beta_{\e}$.
\end{proposition}
\begin{proof}
For each fixed $\e$, if the assumptions in~\eqref{eq1july2} hold, 
then case \eqref{eq1:10feb13} of Proposition~\ref{prop1july20} is satisfied.
Therefore, for each ${\e \in (0,1]}$, there exists 
$\alpha_{\e}$ and $\beta_{\e}$ such that 
${0<\alpha_{\e} \le u_{\e} \le \beta_{\e}}$. 
\end{proof}

%%%%%%%%%%%%%%%%%%%%%%%%%%%%%%%%%%%%%%%
\subsection{Sub- and Super-Solutions}
\label{supersolution}
In the previous section we showed that if the data of \eqref{eq1july1} satisfies \eqref{eq1july2} and if $u_{\e}\in C^{\infty}(\ol{\Om})$ solves
\eqref{eq1july1} for each $\e$,
then $0< \alpha_{\e} \le u_{\e} \le \beta_{\e}$.  Now, for each $\e \in (0,1]$, we want to show that there actually exists a solution $u_{\e}\in C^{\infty}(\ol{\Om})$
satisfying $0<\alpha_{\e} \le u_{\e} \le \beta_{\e}$.  The key to proving this result lies in the fact that $\alpha_{\e}$ and $\beta_{\e}$
are sub- and super-solutions to \eqref{eq1july1} for each $\e$.

\begin{proposition}\label{prop1:15nov11}
%\mnote{Tweaked statement and proof to include case when $\alpha_{\e} = \beta_{\e}$.}
%\mnote{Looks OK. -mjh}
Suppose that the coefficients in the net of problems~\eqref{eq1july1} satisfy~\eqref{eq1july2}. 
Then there exists a net $(u_{\e}) \in (C(\ol{\Om}))^I$ such that for each $\e$, $u_{\e}$ solves~\eqref{eq1july1} and  ${0 < \alpha_{\e} \le u_{\e} \le \beta_{\e}}$,
where $\alpha_{\e}$ and $\beta_{\e}$ be the bounds established in Proposition~\ref{prop2july20}.  . 
\end{proposition}

\begin{proof}
To solve the above family of problems in~\eqref{eq1july1}, we show that the net of $L^{\infty}$-bounds $(\alpha_{\epsilon})$ 
and $(\beta_{\epsilon})$ found in Proposition~\ref{prop2july20} is a net of sub and super-solutions to~\eqref{eq1july1}.  
We then apply Theorem~\ref{thm2june27} to
conclude that for each $\e$, there exists a solution $u_{\e} \in C^{\infty}(\overline{\Om})$.

Fix $\e$ and let $\alpha'_{\epsilon}$ and $\beta'_{\epsilon}$ be defined by~\eqref{eq3:10feb13} and~\eqref{eq4:10feb13} respectively, and 
let 
\begin{align*}
\alpha_{\epsilon} &=\min\{\alpha'_{\epsilon},\inf_{\partial\Omega}\rho_{\epsilon}(x)\}, \\
\beta_{\epsilon} &= \max\{\beta'_{\epsilon},\sup_{x\in\partial\Omega}\rho_{\epsilon}(x)\}.
\end{align*}
The conditions in Eq. \eqref{eq1july2} and the fact that $\rho_{\e} > 0$ imply that $\alpha_{\e} > 0$.
Then the definition of $\alpha_{\epsilon}$ implies that
\begin{align}
A_{\epsilon}\alpha_{\epsilon} &= \sum_{i=1}^K b^i_{\epsilon}(\alpha_{\epsilon})^{n_i} \le \sum_{i=1}^K \sup_{x\in\overline{\Omega}}b^i_{\epsilon}(\alpha_{\epsilon})^{n_i}\le 0,\\
\alpha_{\epsilon} & \le \inf_{x\in\partial\Omega}\rho_{\epsilon}(x) \le \rho_{\epsilon}, \nonumber
\end{align}
which shows that $\alpha_{\epsilon}$ is sub-solution for each $\epsilon$.  Similarly, the conditions in Eq. \eqref{eq1july2} and the definition of $\beta'_{\epsilon}$
imply that
\begin{align}
A_{\epsilon}\beta_{\epsilon} &= \sum_{i=1}^K b^i_{\epsilon}(\beta_{\epsilon})^{n_i} \ge \sum_{i=1}^K \inf_{x\in\overline{\Omega}}b^i_{\epsilon}(\beta_{\epsilon})^{n_i}\ge 0,\\
\beta_{\epsilon} & \ge \sup_{x\in\partial\Omega}\rho_{\epsilon} \ge \rho_{\epsilon},  \nonumber
\end{align}
which shows that $\beta_{\epsilon}$ is a super-solution for each $\epsilon$.

What remains is to show that that $\alpha_{\e} \le \beta_{\e}$.
Given the definition of $\alpha_{\e}$ and $\beta_{\e}$, it suffices to show that
$\alpha'_{\e} \le \beta'_{\e}$.  Define 
$$
\gamma_{\epsilon}=\inf_{c\in\mathbb{R}}\{\sum_{i=1}^K\sup_{x\in\overline{\Omega}}b^i(x)d^{n_i} \ge 0 \hspace{3mm} \forall d\in(c,\infty)\}.
$$
Then we have that $\alpha'_{\epsilon}\le\gamma_{\epsilon}$ by the definition of $\alpha_{\e}'$.
Furthermore, for a fixed $\e$, given the assumptions on $b_{\e}^i(x)$,
$$
\sum_{i=1}^K\inf_{x\in\overline{\Omega}}b_{\e}^i(x)y^{n_i} \le \sum_{i=1}^K\sup_{x\in\overline{\Omega}}b_{\e}^i(x)y^{n_i}, \hspace{3mm} \forall y\in \mathbb{R}.
$$
Therefore the definition of $\beta_{\e}'$ and the above inequality clearly imply that $ \gamma_{\epsilon} \le \beta'_{\epsilon}$.  Therefore $\alpha'_{\epsilon}\le\beta'_{\epsilon}$
and the interval $[\alpha_{\epsilon},\beta_{\epsilon}]$ is a nonempty subset of $\mathbb{R}^+$.  For each $\e \in (0,1]$, the hypotheses of Theorem~\ref{thm2june27} 
are satisfied for the elliptic problem \eqref{eq1july2}, so we may conclude that there exists a net of solutions $(u_{\e}) \in (C^{\infty}(\ol{\Om}))^I$ that satisfy $0<\alpha_{\e} \le u_{\e} \le \beta_{\e}$ for each fixed $\e$. 
\end{proof}

The final task in this section is to show that an operator $A \in \mathcal{A}_0$, with coefficients satisfying \eqref{eq1july2},
is a well-defined operator on any element $u \in \mathcal{E}_M(\ol{\Om})$ satisfying 
$$ \alpha_{\e} \le u_{\e} \le \beta_{\e} \quad \forall \e \in (0,1].$$  
Recall that in Section~\ref{netsofproblems} we determined that $A$ is only well-defined for invertible $u \in \mathcal{G}(\ol{\Om})$.  Therefore, 
it suffices to show that
$(\alpha_{\e}), (\beta_{\e})$ and $(\frac{1}{\alpha_{\e}}),(\frac{1}{\beta_{\e}})$ are generalized constants \eqref{eq1:3nov11}, which we verify
in the following lemma.

\begin{lemma}
\label{lem1july4} 
Let $(\alpha_{\e})$ and $(\beta_{\e})$ be the net of sub- and super-solutions to~\eqref{eq1july1} determined in Section~\ref{bounds}.  
Suppose that the coefficients of~\eqref{eq1july1} satisfy~\eqref{eq1july2}.  Then $(\alpha_{\e})$, $(\beta_{\e})$,$ (\frac{1}{\alpha_{\e}})$, and $(\frac{1}{\beta_{\e}})$ are in $\overline{\mathbb{C}}$, the ring of generalized constants.  
\end{lemma}
\begin{remark}
Note that if  $(\frac{1}{\alpha_{\e}}) \in \ol{\mathbb{C}}$, then this implies that there exists an $\e_0\in (0,1)$, some constant $C$ 
independent of $\e$ and $a\in \mathbb{R}$ such that 
$\alpha_{\e} \ge C\e^a$ for all $\e \in (0,\e_0)$.  Then if $(u_{\e}) \in \mathcal{E}_M(\ol{\Om})$ satisfies $\alpha_{\e} \le u_{\e} \le \beta_{\e}~$  for each $\e$, $(\frac{1}{\alpha_{\e}}) \in \ol{\mathbb{C}}$ implies that $u = [(u_{\e})]$ is invertible in $\mathcal{G}(\ol{\Om})$.  See Section~\ref{netsofproblems} and \cite{GKOS01} for more details.

\end{remark}
\begin{proof}
We need to show
that there exists constants $D_1,D_2$ independent of $\e$ and $\e_0 \in (0,1)$ such that for all $\e \in (0,\e_0)$,
\begin{align*}
\alpha_{\e}&\ge D_1\e^{b_1} \hspace{3mm} \text{for some $b_1\in\mathbb{R}$}, \\
\beta_{\e} &\le D_2\e^{b_2} \hspace{3mm} \text{for some $b_2\in\mathbb{R}.$}
\end{align*}
So it is necessary to verify that there exists constants $D_1$ and $D_2$ so that for $\e$ sufficiently small
\begin{align*}
\alpha'_{\e} &\ge D_1\e^{b_1}, \hspace{3mm} \text{and} \hspace{3mm}\inf_{x\in \partial\Om} \rho_{\e} \ge D_1\e^{b_1},
\\
\beta'_{\e} &\le D_2\e^{b_2}, \hspace{3mm} \text{and} \hspace{3mm}\sup_{x\in\partial\Om} \rho_{\e} \le D_2\e^{b_2}.
\end{align*}
Given that $(\rho_{\e}) \in \mathcal{E}_M(\overline{\Om})$,  
$$\sup_{x\in\partial\Om} \rho_{\e} \le \sup_{x\in\overline{\Om}}\rho_{\e} = \mathcal{O}(\e^b),$$ 
for some $b\in \mathbb{R}$.  
This and the assumption on $(\rho_{\e})$ in \eqref{eq1july2} imply that we only need to obtain the necessary $\e$-bounds on $\alpha'_{\e}$ and $\beta'_{\e}$. 

For now, drop the $\e$ notation and consider $\alpha'$ defined in \eqref{eq3:10feb13}.  For a given function $f$, define
$$ \underline{\gamma}_{f} = \sup_{c\in\mathbb{R}_+}\left\{f(b) \le 0 \hspace{3mm}\forall b\in (0,c)\right\}.$$
Given that 
$$\alpha' = \sup_{c\in\mathbb{R}_+} \left\{\sum_{i=1}^K\sup_{x\in\ol{\Omega}}b^i(x)y^{n_i}\le 0 \hspace{3mm}\forall y\in (0,c)\right\},$$
it is clear that for another function $f(y)$ such that 
$$f(y) \ge \sum_{i=1}^K\sup_{x\in\ol{\Omega}}b^i(x)y^{n_i} \hspace{3mm} \text{on $(0,c)$},$$
if $\underline{\gamma}_f$ is defined and $\alpha' \in (0,c)$, it must hold that $\underline{\gamma}_f \le \alpha'$.  
Let $C_1 =|\{n_i:n_i\ge0\}|$ and $C_2 =|\{n_i:n_i<0\}|$ and if $C_2 > 1$,
 let $n_{i_2} =\min\{n_i: n_1<n_i<0\}$.  Note that $C_1, C_2 \ge 1$ based on the assumptions in~\eqref{eq1july2}.   
Then recalling that $b_1(x)<0$,  $b_K(x)>0$ correspond to the coefficients of the terms with the smallest negative and largest positive exponent of $\sum^K_i b^i(x)u^{n_i}$,
 if $\sup_{x\in\overline{\Om}}|b^i(x)| \le \Lambda$ for each $i$, the following must hold for $y\in (0,1)$:
\begin{align}
\sum_{i=1}^K\sup_{x\in\ol{\Omega}}b^i(x)y^{n_i}\le \sup_{x\in\overline{\Omega}}b_1(x)y^{n_1}+C_1\Lambda +(C_2-1)\Lambda y^{n_{i_2}}.
\end{align}

Define
$$
d=\left(\frac{-\sup_{x\in\ol{\Om}} (b_1(x))}{2(C_2-1)\Lambda}\right)^{\frac{1}{n_{i_2}-n_1}}
$$
if $ C_2 > 1$ and let $d =1$ if $C_2 =1$.  Then let $c = \min\{1,d\}$.
The definition of $c$ implies that 
$$(C_2-1)\Lambda y^{n_{i_2}}\le -\frac{\sup_{x\in\overline{\Om}}b_1(x)}{2} y^{n_1},$$ 
for all $y\in (0,c)$.  So for $y\in (0,c),$
$$
\sum_{i=1}^K\sup_{x\in\overline{\Omega}}b^i(x)y^{n_i}\le \frac{\sup_{x\in\overline{\Om}}b_1(x)}{2} y^{n_1} + C_1\Lambda =f(y).
$$
Then if $\alpha' \in (0,c)$, $\alpha' \ge \underline{\gamma}_f$. Given that $f(y)$ is a monotone increasing function on 
$\mathbb{R}_+$, $\underline{\gamma}_f$ is the lone positive root of $f(y)$.  Thus, 
$$\underline{\gamma}_f = \left(\frac{-\sup_{x\in\ol{\Om}}b_1(x)}{2C_1\Lambda}\right)^{\frac{1}{-n_1}},$$ 
which implies that if $\alpha' \in (0,c)$,
$$\alpha' \ge \left(\frac{-\sup_{x\in\ol{\Om}}b_1(x)}{2C_1\Lambda}\right)^{\frac{1}{-n_1}}.$$ 

Similarly, for a fixed $\e \in (0,1)$, define
$$
d_{\e}=\left(\frac{-\sup_{x\in \ol{\Om}}( b_{\e}^1(x))}{2(C_2-1)\Lambda_{\e}}\right)^{\frac{1}{n_{i_2}-n_1}},
$$
if $C_2 > 1 $ and 
let $d_{\e} = 1$ if $C_2 =1$.  Let $c_{\e} = \min\{1,d_{\e}\}$.
Then for $y \in (0,c_{\e})$, we have that 
$$(C_2-1)\Lambda_{\e} y^{n_{i_2}}\le -\frac{\sup_{x\in\overline{\Om}}b_{\e}^1(x)}{2} y^{n_1}.$$
So the above arguments imply that if $\alpha_{\e}' \in (0,c_{\e})$, then $\alpha_{\e}' \ge \underline{\gamma}_{f,\e}$ and
$$\alpha'_{\e} \ge \left(\frac{-\sup_{x\in \ol{\Om}}b^1_{\e}(x)}{2C_{1}\Lambda_{\e}}\right)^{\frac{1}{-n_1}}.$$
Given the assumptions on $b_{\e}^1(x)$ and $\Lambda_{\e}$ in~\eqref{eq1july2}, in this case we have that $\alpha'_{\e} \ge C\e^a$ for
some constant $C>0$, $a\in \mathbb{R}$ and $\e$ sufficiently small. 
Now we must show that $c_{\e} \ge C\e^a$ for some constant $C>0$, $a\in \mathbb{R}$ and $\e$ sufficiently small in the event that $\alpha'_{\e} \notin (0,c_{\e})$.  
It suffices to show that $d_{\e} \ge C\e^a$ in the event that $C_2 > 1$.
But clearly, for $\e$ sufficiently small
$$
d_{\e} = \left(-\frac{\sup_{x\in \ol{\Om}}b^1_{\e}(x)}{2(C_2-1)\Lambda_{\e}}\right)^{\frac{1}{n_{i_2}-n_1}} \ge C\e^{a},
$$
given the assumptions on $b^1_{\e}$ and $\Lambda_{\e}$ in~\eqref{eq1july2}. 
Therefore $\alpha'_{\e} \ge D_1\e^a$ for some constant $D_1>0$, $a \in \mathbb{R}$ and $\e$ sufficiently small.  
  
Now we determine bounds on the net $(\beta'_{\e})$.  Again, we temporarily drop the $\e$ and only consider $\beta'$.  Recall that
$$\beta' = \inf_{c\in\mathbb{R}} \left\{\sum_{i=1}^K\inf_{x\in\ol{\Omega}}b^i(x)y^{n_i}\ge 0 \hspace{3mm}\forall y\in (c,\infty)\right\}.$$
For a given function $f(y)$, define 
$$
\overline{\gamma}_f=\inf_{c\in\mathbb{R}} \left\{f(b) \ge 0 \hspace{3mm}\forall b\in (c,\infty)\right\}.
$$
Then if $f(y) \le \sum_{i=1}^K\sup_{x\in\ol{\Omega}}b^i(x)y^{n_i}$ on some interval $(c,\infty)$ and 
$\beta' \in (c,\infty)$, it must hold that $\overline{\gamma}_f \ge \beta'$ if $\ol{\gamma}_{f}$ is defined.  Let 
$C_1,C_2$ be as before and let ${n_{i_1} = \max\{n_i:0 \le n_i<n_K\}}$ if $C_1 > 1$. If $y>1$, then
$$
 \sum_{i=1}^K\inf_{x\in\ol{\Omega}}b^i(x)y^{n_i}\ge \inf_{x\in\ol{\Om}}( b_K(x))y^{n_K}-(C_1-1)\Lambda y^{n_{i_1}}-C_2\Lambda.
$$  
Now define 
$$
d = \left(\frac{2(C_1-1)\Lambda}{\inf_{x\in\ol{\Om}} (b_K(x))}\right)^{\frac{1}{n_k-n_{i_1}}}
$$
if $C_1 > 1$ and let $d = 1$ if $C_1 =1$.  Let $c = \max\{1,d\}$.  
Then our choice of $d$ ensures that if $C_1>1$, then
$$
-(C_1-1)\Lambda y^{n_{i_1}} \ge -\frac{\inf_{x\in\ol{\Om}}(b_K(x))y^{n_K}}{2},
$$
and that for $y \in (c,\infty)$,
$$
\sum_{i=1}^K\sup_{x\in\ol{\Omega}}b^i(x)y^{n_i}\ge \frac{\inf_{x\in\ol{\Om}}( b_K(x))}{2}y^{n_K}-C_2\Lambda =f(y).
$$
So if $\beta' \in (c,\infty)$, $\beta' \le  \overline{\gamma}_f$, where $\overline{\gamma}_f$ is the lone positive root of $f$ on $\mathbb{R}_+$ given
that $f$ is monotone increasing on this interval.  So if $\beta' \in (c, \infty)$, 
$$
\beta' \le  \overline{\gamma}_f= \left(\frac{2C_2\Lambda}{\inf_{x\in\ol{\Om}}( b_K(x))}\right)^{\frac{1}{n_K}}.
$$
By defining 
\begin{align}
d_{\e} = \left(\frac{2(C_1-1)\Lambda_{\e}}{\inf_{x\in\ol{\Om}} b_{\e}^K(x)}\right)^{\frac{1}{n_k-n_{i_1}}}, \quad \text{and} \quad
c_{\e} = \max\{1,d_{\e}\},
\end{align}
and applying the above argument for $\beta'$ to the net $(\beta'_{\e})$ for each fixed $\e$, it is clear that if $\beta_{\e}' \in (c_{\e},\infty)$,
then
$$
\beta'_{\e} \le \left(\frac{2C_2\Lambda_{\e}}{\inf_{x\in\ol{\Om}} b_{\e}^K(x)}\right)^{\frac{1}{n_K}} \le C\e^a,
$$
given the assumptions on $b^K_{\e}$ and $\Lambda_{\e}$ in~\eqref{eq1july2}. 

Now assume that $\beta'_{\e} \notin (c_{\e},\infty)$.  Then it suffices to show that if $C_1>1$, 
then $d_{\e} \le C\e^a$ for $\e$ sufficiently small and some positive constants $C$ and $a \in \mathbb{R}$.   
But again, this is clearly true given the assumptions~\eqref{eq1july2} and the fact that
$$
d_{\e} = \left(\frac{2(C_1-1)\Lambda_{\e}}{\inf_{x\in\ol{\Om}} b_{\e}^K(x)}\right)^{\frac{1}{n_k-n_{i_1}}}.
$$
\end{proof}

%%%%%%%%%%%%%%%%%%%%%%%%%%%%%%%%%%%%%%%%%%%%%%%%%%%%%%%%%%%%%%%%%%%%%%%%%%%%%%
\section{Proof of the Main Results}
\label{results}

We now prove Theorem~\ref{thm1june27} using the results from Section~\ref{bounds1}.  For clarity, we break the proof up into the steps outlined in Section~\ref{over}.

%%%%%%%%%%%%%%%%%%%%%%%%%%%%%%%%%%%%%%%
\subsection{Proof of Theorem~\ref{thm1june27}}\label{proofofthm}

\begin{proof}

\begin{itemize}
\item[Step 1:]  { \it Formulation of the problem.}  For convenience, we restate the problem and the formulation that we will use to find a solution.
Given an operator $A \in \mathcal{A}_0$, defined by~\eqref{eq1june26},
we want to solve the following Dirichlet problem in $\mathcal{G}(\overline{\Om})$:
\begin{align}
\label{eq1july4}
Au &= 0 \hspace{3mm} \text{in $\Om$},\\
u  &= \rho  \quad \text{on $\partial\Om$} \nonumber.
\end{align}

We phrase~\eqref{eq1july4} in a way that allows us to solve a net of semilinear
elliptic problems.  We assume that the coefficients of $A$ and boundary data $\rho$ have representatives $(a^{ij}_{\e}), (b^i_{\e}),$ and $(\rho_{\e})$
in $\mathcal{E}_M(\overline{\Om})$ satisfying the assumptions~\eqref{eq1june27}.  Then for this particular choice of representatives, our
strategy for solving~\eqref{eq1july4} is to solve the family of problems
\begin{align}
\label{eq2july4}
A_{\epsilon}u_{\epsilon} = -\sum_{i,j=1}^N D_i( a_{\epsilon}^{ij}D_j u_{\epsilon}) + \sum_i^N b^i_{\epsilon}u_{\epsilon}^{n_i} &= 0 \text{  in  $\Omega$},\\
u_{\epsilon}&= \rho_{\epsilon}  \quad \text{on $\partial\Om$} \nonumber,
\end{align}
and then show that the net of solutions $(u_{\e}) \in \mathcal{E}_M(\ol{\Om})$.

\item[Step 2:] {\it Determine $L^{\infty}$-estimates and a net of sub-solutions and super-solutions}.  In Section~\ref{bounds1}, we concluded that 
for each $\e$, the pair $\alpha_{\e}$ and $\beta_{\e}$ determine sub- and super- solutions to~\eqref{eq2july4} such that $0<\alpha_{\e}<\beta_{\e}$.  
Furthermore, in Lemma~\ref{lem1july4} we concluded that there exist $C_1,C_2>0$ and $a_1,a_2 \in \mathbb{R}$ such that for $\e$ sufficiently small, the nets $(\alpha_{\e})$ and $(\beta_{\e})$ satisfy
$C_1\e^{a_1}\le \alpha_{\e} < \beta_{\e} \le C_2\e^{a_2}$, thereby verifying that $(\alpha_{\e}), (\beta_{\e}), (\frac{1}{\alpha_{\e}}), (\frac{1}{\beta_{\e}}) \in \overline{\mathbb{C}}$,
the ring of generalized constants.

\item[Step 3:]{\it Apply fixed-point theorem to solve each semilinear problem in~\eqref{11feb15eq1}.}  This follows from 
Proposition~\ref{prop1:15nov11}.  We briefly reiterate the proof here.  We simply verify the hypotheses of Theorem~\ref{thm2june27}.  
For each fixed $\e$ we have sub- and super-solutions $\alpha_{\e}$
and $\beta_{\e}$ satisfying $0<\alpha_{\e}<\beta_{\e}$ and $a^{ij}_{\e}, b^i_{\e}, \rho_{\e} \in C^{\infty}(\overline{\Om})$ 
satisfying~\eqref{eq1july2}.  Finally, $\Om$ is of $C^{\infty}$-class and the function 
$$f(x,y) = -\sum_{i=1}^Kb^i_{\e}(x)y^{n_i} \in C^{\infty}(\overline{\Om}\times \mathbb{R}^+),$$
so we may apply Theorem~\ref{thm2june27} to conclude that there exists a net of solutions $(u_{\e})$ to~\eqref{eq1july1} satisfying $0<\alpha_{\e} \le u_{\e} \le \beta_{\e}$.

\item[Step 4:]{\it Verify that the net of solutions $(u_{\e}) \in \mathcal{E}_M(\overline{\Om}).$}  Now that it is clear that a solution exists for~\eqref{eq1july1}
for each $\e \in (0,1]$, it is necessary to establish estimates
that show that the net of solutions $(u_{\e})$ is in $\mathcal{E}_M(\overline{\Omega})$.  That is, we want to show that for each $k \in \mathbb{N}$ and
all multi-indices $|\beta | \le k$, there exists $a \in \mathbb{R}$ such that
$$
\sup_{x \in \overline{\Om}}\{|D^{\beta}u_{\e}(x)|\} = \mathcal{O}(\e^a).
$$
By standard interpolation inequalities, it suffices to show that for $\gamma \in (0,1)$ and each $k \in \mathbb{N}$, there 
exists an $a \in \mathbb{R}$ such that
$$
|u_{\e}|_{k,\gamma;\Om} =  \mathcal{O}(\e^a).
$$
By Theorem~\ref{thm1june30}, we have that if $u_{\e}$ is a solution to~\eqref{eq1july1} with coefficients satisfying~\eqref{eq1july2}, then
\begin{align}
|u_{\e}|_{2,\gamma;\Om}\le C\left(\frac{\Lambda_{\e}}{\lambda_{\e}}\right)^3 (|u_{\e}|_{0;\Om}+|\rho_{\e}|_{2,\gamma;\Om}+\sum_{i=1}^K |b^i_{\e}(u_{\e})^{n_i}|_{0,\gamma;\Om}).
\end{align}
Observe that
\begin{align}
\label{eq1july5}
|u_{\e}^{n_i}|_{0,\gamma;\Om} \le |u^{n_i}_{\e}|_{0;\Om} + n_i[u_{\e}]_{0,\gamma;\Om}|u_{\e}|^{n_i-1}_{0;\Om}
\end{align}
if $n_i>0$ and
\begin{align}
\label{eq2july5}
|u_{\e}^{n_i}|_{0,\gamma;\Om} \le |u^{n_i}_{\e}|_{0;\Om} + \frac{1}{|u_{\e}^{-n_i}|^2_{0;\Om}}(-n_i)[u_{\e}]_{0,\gamma;\Om}|u_{\e}|^{-n_i-1}_{0;\Om},
\end{align}
if $n_i<0$.  The above inequality implies that
\begin{align}
|u_{\e}|_{2,\gamma;\Om}&\le C\left(\frac{\Lambda_{\e}}{\lambda_{\e}}\right)^3(|u_{\e}|_{0;\Om}+|\rho_{\e}|_{2,\gamma;\Om} \\
& \qquad +\sum_{i=1}^K |b^i_{\e}(x)|_{0,\gamma;\Om}(C_1(\alpha_{\e},\beta_{\e},n_i)+C_2(n_i,\alpha_{\e},\beta_{\e})|u_{\e}|_{0,\gamma;\Om})),  \nonumber
\end{align}
where 
$$C_1(n_i,\alpha_{\e},\beta_{\e}) = \beta_{\e}^{n_i} \hspace{2mm} \text{ and} \hspace{2mm} C_2(n_i,\alpha_{\e},\beta_{\e}) = n_i\beta_{\e}^{n_i-1}, \hspace{2mm} \text{ if} \hspace{2mm} n_i>0 \hspace{2mm} \text{ and}$$ 
$$C_1(n_i,\alpha_{\e},\beta_{\e}) = \alpha_{\e}^{n_i} \hspace{2mm} \text{  and} \hspace{2mm} C_2(n_i,\alpha_{\e},\beta_{\e})=\frac{(-n_i)\beta_{\e}^{-n_i-1}}{\alpha_{\e}^{-2n_i}} \hspace{2mm} \text{ if} \hspace{2mm} n_i<0.$$
Application of the interpolation inequality 
$$|u_{\e}|_{0,\gamma} \le C(\delta_{\e}^{-1}|u_{\e}|_0 + \delta_{\e}|u_{\e}|_{2,\gamma}),$$ 
where $\delta_{\e}$ is arbitrarily small and $C$ is independent of $\delta_{\e}$, implies that
\begin{align}
|u_{\e}|_{2,\gamma;\Om}
& \le C\left(\frac{\Lambda_{\e}}{\lambda_{\e}}\right)^3(|u_{\e}|_{0;\Om}
      +|\rho_{\e}|_{2,\gamma;\Om}
\\
&\qquad +\sum_{i=1}^K |b^i_{\e}(x)|_{0,\gamma;\Om}(C_1(n_i,\alpha_{\e},\beta_{\e})
\nonumber \\
& \qquad +C_2(n_i,\alpha_{\e},\beta_{\e})(C(\delta_{\e}^{-1}|u_{\e}|_{0;\Om}+\delta_{\e}|u_{\e}|_{2,\gamma;\Om}) ))).  \nonumber
\end{align}
Therefore,
\begin{align}
&\left(1-\delta_{\e}\left( \frac{\Lambda_{\e}}{\lambda_{\e}} \right) \sum_{i=1}^K |b^i_{\e}(x)|_{0,\gamma;\Om}C_2(n_i,\alpha_{\e},\beta_{\e})\right)|u_{\e}|_{2,\gamma;\Om} \\ 
& \qquad \quad \le C\left(\frac{\Lambda_{\e}}{\lambda_{\e}}\right)^3(|u_{\e}|_{0;\Om}+|\rho_{\e}|_{2,\gamma;\Om} \nonumber \\
& \qquad \quad \quad +\sum_{i=1}^K |b^i_{\e}(x)|_{0,\gamma;\Om}( C_1(n_i,\alpha_{\e},\beta_{\e})+C_2(n_i,\alpha_{\e},\beta_{\e})\delta_{\e}^{-1}|u_{\e}|_{0;\Om})).  \nonumber
\end{align}
But given the assumptions on $\Lambda_{\e}$, $\lambda_{\e}$, the bounds previously established for the nets $(\alpha_{\e}) \text{ and } (\beta_{\e})$ in Lemma~\ref{lem1july4}, and given that $(b^i_{\e}(x)) \in \mathcal{E}_M(\overline{\Om})$, there exists $\e_0 \in (0,1)$, $a \in \mathbb{R}$ and $C>0$ such that for all $\e \in (0,\e_0)$,
$$
\left( \frac{\Lambda_{\e}}{\lambda_{\e}} \right) \sum_{i=1}^K |b^i_{\e}(x)|_{0,\gamma}C_2(n_i,\alpha_{\e},\beta_{\e}) \le C\e^a .
$$
Therefore, choosing 
$$
\delta_{\e} = \frac{1}{2C\e^a},
$$
it is clear that for $\e \in (0,\e_0)$,
\begin{align}
|u_{\e}|_{2,\gamma;\Om} &\le C\left(\frac{\Lambda_{\e}}{\lambda_{\e}}\right)^3(|u_{\e}|_{0;\Om}+|\rho_{\e}|_{2,\gamma;\Om} \\
& \qquad + \sum_{i=1}^K |b^i_{\e}(x)|_{0,\gamma;\Om}( C_1(n_i,\alpha_{\e},\beta_{\e})+C_2(n_i,\alpha_{\e},\beta_{\e},\e^a)|u_{\e}|_{0;\Om})). \nonumber
\end{align}
Given that $(\alpha_{\e})$, $(\beta_{\e}) \in \ol{\mathbb{C}}$,  $\alpha_{\e}\le u_{\e}\le \beta_{\e}$ and $(\rho_{\e})$, $ (b^i_{\e}) \in \mathcal{E}_M(\overline{\Om})$, the above inequality implies that for some $a\in \mathbb{R}$,
$$
|u_{\e}|_{2,\gamma;\Om} = \mathcal{O}(\e^a).
$$  
Now we need to utilize the $\e$-growth conditions on $|u_{\e}|_{2,\gamma;\Om}$ and induction to show that for any $k >2$ that 
\begin{align}
\label{eq2july4-B}
|u_{\e}|_{k,\gamma;\Om}=\mathcal{O}(\e^a) \quad \text{for some $a\in \mathbb{R}$}.
\end{align}
Let $(u_{\e})$ be a smooth net of solutions to~\eqref{eq2july4} and additionally assume that \eqref{eq2july4-B} holds for
all $j\le k$.  Let $\nu$ be a multi-index of length $k-1$.  Then by differentiating both sides of~\eqref{eq2july4}, we see that for each $\e$, $u_{\e}$ satisfies the Dirichlet problem
\begin{align}
\label{eq3uly4}
\sum_{i,j=1}^N D^{\nu}( -D_i( a_{\epsilon}^{ij}D_j u_{\epsilon})) &=  -\sum_{i=1}^K D^{\nu}( b^i_{\epsilon}u_{\epsilon}^{n_i}) \text{  in  $\Omega$} \\
D^{\nu}u_{\epsilon} &= D^{\nu}\rho_{\epsilon}  \quad \text{on $\partial\Om$} . \nonumber
\end{align} 
Rearranging the above equation and applying the multi-index product rule we find that
\begin{align}
\label{eq4july4}
\sum_{i,j=1}^N a^{ij}_{\e}D_{ij}(D^{\nu}u_{\e}) &= -\sum_{i,j=1}^N D^{\nu}((D_ia_{\e}^{ij})(D_ju_{\e}))
\\
& \qquad 
-\sum_{i,j=1}^N\sum_{\stackrel{\sigma+\mu = \nu}{\sigma \ne \nu}}\frac{\nu!}{\sigma!\mu!}(D^{\mu}a_{\e}^{ij})(D^{\sigma}D_{ij}u_{\e})
\nonumber \\
& \qquad
+\sum_{i=1}^K\sum_{\sigma+\mu = \nu}\frac{\nu!}{\sigma!\mu!}(D^{\mu}b^i_{\e})(D^{\sigma}((u_{\e})^{n_i})). \nonumber
\end{align}
Therefore, we may apply Theorem~\ref{thm1june30} to~\eqref{eq4july4} to conclude that for an arbitrary multi-index $\nu$
such that $|\nu| = k-1$,
\begin{align}
|D^{\nu}u_{\e}|_{2,\gamma;\Om} 
&\le 
   C\left(\frac{\Lambda_{\e}}{\lambda_{\e}}\right)^3
   \left(|D^{\nu}u_{\e}|_{0;\Om} + |D^{\nu}\rho_{\e}|_{2,\gamma;\Om} \right.
   \\
& \qquad + |\sum_{i,j=1}^N D^{\nu}((D_ia_{\e}^{ij})(D_ju_{\e}))|_{0,\gamma;\Om} 
   \nonumber \\
& \qquad +\sum_{i,j=1}^N \sum_{\stackrel{\sigma+\mu = \nu}{\sigma \ne \nu}}\frac{\nu!}{\sigma!\mu!}|D^{\mu}a_{\e}^{ij}|_{0,\gamma;\Om}|D^{\sigma}D_{ij}u_{\e}|_{0,\gamma;\Om}
  \nonumber \\
& \qquad
  +\sum_{i=1}^K\sum_{\sigma+\mu = \nu}\frac{\nu!}{\sigma!\mu!}|D^{\mu}b^i_{\e}|_{0,\gamma;\Om}|D^{\sigma}((u_{\e})^{n_i})|_{0,\gamma;\Om}
  ) \nonumber \\
& \le 
  C\left(\frac{\Lambda_{\e}}{\lambda_{\e}}\right)^3(|D^{\nu}u_{\e}|_{0;\Om} 
+ |D^{\nu}\rho_{\e}|_{2,\gamma;\Om}
  \nonumber \\
& \qquad
+\sum_{i,j=1}^N \sum_{\sigma+\mu = \nu}\frac{\nu!}{\sigma!\mu!}|D^{\mu}(D_ia_{\e}^{ij})|_{0,\gamma;\Om}|D^{\sigma}(D_ju_{\e})|_{0,\gamma;\Om}
  \nonumber \\
& \qquad
  +\sum_{i,j=1}^N \sum_{\stackrel{\sigma+\mu = \nu}{\sigma \ne \nu}}\frac{\nu!}{\sigma!\mu!}|D^{\mu}a_{\e}^{ij}|_{0,\gamma;\Om}|D^{\sigma}D_{ij}u_{\e}|_{0,\gamma;\Om}
  \nonumber \\
& \qquad +\sum_{i=1}^K \sum_{\sigma+\mu = \nu}\frac{\nu!}{\sigma!\mu!}|D^{\mu}b^i_{\e}|_{0,\gamma;\Om}|D^{\sigma}((u_{\e})^{n_i})|_{0,\gamma;\Om}     ).
\nonumber
\end{align}
By our inductive hypothesis and the assumptions on the coefficients, it is immediate that every term in the above expression is $\mathcal{O}(\e^a)$ for some $a \in \mathbb{R}$ except for the last term.
So to show
$$
|D^{\nu}u_{\e}|_{2,\gamma;\Om} = \mathcal{O}(\e^a) \quad \text{for some $a \in \mathbb{R}$},
$$
it suffices to show that 
$$
\sum_{i=1}^K \sum_{\sigma+\mu = \nu}\frac{\nu!}{\sigma!\mu!}|D^{\mu}b^i_{\e}|_{0,\gamma;\Om}|D^{\sigma}((u_{\e})^{n_i})|_{0,\gamma;\Om} = \mathcal{O}(\e^a) \quad \text{for some $a \in \mathbb{R}$}.
$$
Given that $b^i_{\e} \in \mathcal{E}_M(\ol{\Om})$ for each $1 \le i \le K$,
$$
|D^{\mu}b^i_{\e}|_{0,\gamma;\Om} =  \mathcal{O}(\e^a) \quad \text{for some $a\in \mathbb{R}$}.
$$
Therefore, it is really only necessary to show
that for any multi-index $\sigma$, such that $|\sigma| = j \le k-1$, that there exists an $a \in \mathbb{R}$
such that
$$
|D^{\sigma}((u_{\e})^{n_i})|_{0,\gamma;\Om} = \mathcal{O}(\e^a).
$$
But observe that 
$D^{\sigma}((u_{\e})^{n_i})$ is a sum of terms of the form
$$
(u_{\e})^{n_i-m}D^{\sigma_1}u_{\e}D^{\sigma_2}u_{\e}\cdots D^{\sigma_m}u_{\e},
$$
where $\sigma_1+\sigma_2 +\cdots  \sigma_m = \sigma$ and $m \le j \le k-1$.
This follows immediately from the chain rule.
Therefore we have the following bound:
\begin{align}
|D^{\sigma}((u_{\e})^{n_i})|_{0,\gamma;\Om} 
& \le (n_i)|(u_{\e})^{n_i-1}|_{0,\gamma;\Om}|D^{\sigma}u_{\e}|_{0,\gamma;\Om}
   \\
& \qquad
   +\sum_{\sigma_1+\sigma_2=\sigma}\frac{\sigma!}{\sigma_1!\sigma_2!}(n_i)(n_i-1)|(u_{\e})^{n_i-2}|_{0,\gamma;\Om} \nonumber \\
& \qquad \qquad 
  \cdot |D^{\sigma_1}u_{\e}|_{0,\gamma;\Om}|D^{\sigma_2}u_{\e}|_{0,\gamma;\Om}+ \cdots 
   \nonumber \\
& \qquad +\sum_{\sigma_1+\sigma_2+\cdots+\sigma_j = \sigma}\frac{\sigma!}{\sigma_1!\sigma_2!\cdots \sigma_j!}(n_i)(n_i-1)
   \nonumber \\
& \qquad \qquad
   \cdots(n_i-j) |(u_{\e})^{n_i-j}|_{0,\gamma;\Om} |D^{\sigma_1}u_{\e}|_{0,\gamma;\Om}
   \nonumber \\
& \qquad \qquad
\cdots |D^{\sigma_j}u_{\e}|_{0,\gamma;\Om}.
   \nonumber
\end{align}
Using~\eqref{eq1july5} and~\eqref{eq2july5}, for each $m \le j$ we may bound the terms of the form $|(u_{\e})^{n_i-m}|_{0,\gamma;\Om}$ using $|u_{\e}|_{0, \gamma;\Om}$, $\alpha'_{\e}$ and $\beta'_{\e}$. 
Then our inductive hypothesis and the growth conditions on $(\alpha'_{\e})$ and $(\beta'_{\e})$ imply that 
$$
|D^{\sigma}((u_{\e})^{n_i})|_{0,\gamma;\Om} =  \mathcal{O}(\e^a) \quad \text{for some $a \in \mathbb{R}$}
$$
This implies that 
$$
|D^{\nu}u_{\e}|_{2,\gamma;\Om} = \mathcal{O}(\e^a) \quad \text{ for some $a \in \mathbb{R}$}.
$$
As $\nu$ was an arbitrary multi-index such that $|\nu| = k-1$, this implies there exists $a\in \mathbb{R}$ such that 
$$
|u_{\e}|_{k+1,\gamma;\Om} = \mathcal{O}(\e^a).
$$
Therefore, $(u_{\e}) \in \mathcal{E}_M(\overline{\Om})$. 

\item[Step 5:]{\it Verify that the solution is well-defined}.   
Proposition~\ref{prop1july1} and the definition of the Dirichlet problem in $\mathcal{G}(\overline{\Om})$ given in Section~\ref{Dirichlet} imply that 
$[(u_{\e})]$ is indeed a solution to the problem
\begin{align}
\label{welldefined}
Au &= 0 \hspace{3mm} \text{in $\Om$},\\
u &= \rho  \quad \text{on $\partial\Om$} , \nonumber
\end{align}
in $\mathcal{G}(\overline{\Om})$.  To see this, we consider other representatives $(\overline{a}_{\e}^{ij}), (\overline{b}^i_{\e}), (\overline{\rho}_{\e})$, and $(\overline{u}_{\e})$ 
of $[(a^{ij}_{\e})], [(b^i_{\e})], [(\rho_{\e})]$, and $[(u_{\e})]$.  Then the proof of Proposition~\ref{prop1july1} clearly implies that 
\begin{align}
-\sum_{i,j=1}^N D_i(\overline{a}_{\e}^{ij}D_j\overline{u}_{\e}) &+ \sum_{i=1}^K \overline{b}^i_{\e}(\overline{u}_{\e})^{n_i} = \eta_{\e} \hspace{3mm} \text{in $\Om$},\\
\overline{u_{\e}} &= \overline{\rho}_{\e}+\overline{\eta}_{\e}  \quad \text{on $\partial\Om$} ,\nonumber
\end{align}
where $\eta_{\e} \in \mathcal{N}(\overline{\Om})$ and $\overline{\eta}_{\e}$ is a net of functions satisfying~\eqref{boundary}.  
But this implies that this choice of representatives also satisfies~\eqref{welldefined} in $\mathcal{G}(\overline{\Om})$, so our
solution $[(u_{\e})]$ is independent of the representatives used.
\end{itemize}
\end{proof}

This completes our proof of Theorem~\ref{thm1june27}.
We now conclude by giving a brief summary and making some final remarks.

%%%%%%%%%%%%%%%%%%%%%%%%%%%%%%%%%%%%%%%%%%%%%%%%%%%%%%%%%%%%%%%%%%%%%%%%%%%%%%
\section{Summary and Remarks}
\label{sec:conc}
%\mnote{Haven't touched summary.  It is possible that this could use some changes in light of changes made earlier in paper.}
%\mnoteR{I am still reading this. -mjh}

We began the paper with an example to motivate the Colombeau Algebra method 
for solving the target semilinear problem \eqref{problem} with potentially 
distributional data.
In particular, in Section~\ref{Example1} we proved the existence of a solution
to a simpler ill-posed critical exponent problem~\eqref{eq3:25oct11}
in Proposition~\ref{prop1:8nov11}.
Our proof technique consisted of mollifying the data of the original problem,
and then solving a sequence of "approximate" problems with the smooth 
coefficients.
We then obtained a sequence of solutions that yielded a convergent subsequence.
This proof framework, which required only basic elliptic PDE theory, was 
modeled on the more general Colombeau approach that we then subsequently 
developed and applied in the remainder of the paper to
solve the more difficult problem~\eqref{problem}.
Following the approach of Mitrovic and Pilipovic in~\cite{MP06},
in Section~\ref{prelim} we stated a number of preliminary results 
and developed necessary technical tools for solving~\eqref{problem}.
Among these tools and results were the explicit {\em a priori} estimates found 
in~\cite{MP06}, and a description of the Colombeau framework in which the 
coefficients and data were embedded.
In particular, in Section~\ref{holder} we introduced notation for H\"{o}lder
norms and stated two {\em a priori} estimates from~\cite{GiTr77} that were 
made more precise by Mitrovic and Pilipovic in~\cite{MP06}.
In Section~\ref{Colombeau}, we then introduced the general framework for
constructing Colombeau-type algebras and the Colombeau algebra 
$\mathcal{G}(\overline{\Om})$ used in this paper.

We then stated the main result in Section~\ref{overview}, 
namely Theorem~\ref{thm1june27}, and also gave a statement and proof 
of the method of sub- and super solutions as Theorem~\ref{thm2june27}.
We then gave a detailed outlined of the plan of the proof of 
Theorem~\ref{thm1june27}, the execution of which was the focus of the 
remainder of the paper.
In Section~\ref{embed2} we also discussed methods to 
embed \eqref{problem} into the algebra for applying our 
Colombeau existence theory.
%We then discussed a method used to embed the Schwartz distributions 
%$\mathcal{D}'(\Om)$ into $\mathcal{G}(\overline{\Om})$.
%We used this embedding to analyze a problem of the form \eqref{problem} with 
%distributional coefficients.
%In particular, in Proposition~\ref{eq1:24feb12} we determined explicit 
%conditions under which we could solve a semilinear problem of the 
%form \eqref{problem} with rough coefficients.  
%Then we finished Section~\ref{Colombeau} by defining a class of semilinear 
%operators on $\mathcal{G}(\ol{\Om})$ in \ref{netsofproblems}, and we then 
%defined the Dirichlet problem for these operators.
The remainder of the paper was then dedicated to developing the remaining
tools necessary to proving Theorem~\ref{thm1june27}, and then carrying out
the proof.
In Section~\ref{bounds1} we determine {\em a priori} $L^{\infty}$ bounds of 
solutions to our semilinear problem and a net of sub- and super-solutions 
satisfying explicit $\e$-growth estimates.
We first determined a net of $L^{\infty}$ bounds for positive solutions to 
our problem.
In Section~\ref{supersolution} we then showed that this net of $L^{\infty}$ 
bounds is in fact a net of sub- and super-solutions contained in 
$\mathbb{\overline{C}}$, the ring of generalized constants described in 
Section~\ref{Colombeau}.
Finally, after developing sub- and super-solutions
and some related results in Section~\ref{bounds1},
we proved the main result, Theorem~\ref{thm1june27} in Section~\ref{results},
following the plan we had laid out in Section~\ref{overview}.

We note that although the problem we set up in a manner similar to that 
used by Mitrovic and Pilipovic in~\cite{MP06},
our approach to solving our semilinear problem was distinct from 
theirs; we first determined a net of solutions $(u_{\e})$ to the family of 
semilinear problems~\eqref{eq2july4} by using the method of sub-and 
super-solutions (Theorem~\ref{thm2june27}), and our net of sub- and 
super-solutions determined in Section~\ref{supersolution}.
Once our net of solutions was determined, we then employed 
Theorems~\ref{thm1june30} and our net of sub- and super-solutions to show 
that our net of solutions was contained in $\mathcal{E}_M(\overline{\Om})$.

In this article we have attempted to develop some basic tools to allow
for a more general study of the Einstein constraint equations with 
distributional data.
Our goal was to extend the current solution theory for scalar, critical 
exponent semilinear problems such as the Lichnerowicz equation, allowing 
for more irregular data than is currently covered by the existing 
solutions theories (cf.~\cite{HNT08,HNT09} for a summary of the known
results for the CMC, near-CMC, and Far-CMC cases through 2009).
As a next step, we hope to use the tools developed in this article
to extend the near-CMC and Far-CMC existence framework for rough metrics 
developed in \cite{HNT08,DMa05,DMa06,CB04}
to cover the rough data example studied by 
Maxwell in \cite{DMa11}.

%%%%%%%%%%%%%%%%%%%%%%%%%%%%%%%%%%%%%%%%%%%%%%%%%%%%%%%%%%%%%%%%%%%%%%%%%%%%%%
\section*{Acknowledgments}
\label{sec:ack}

MH was supported in part by NSF Awards~1217175 and 1065972.
CM was supported in part by NSF Award~1065972.
%\mnoteR{Mike has to get the correct award numbers from Anna. -mjh}

%%%%%%%%%%%%%%%%%%%%%%%%%%%%%%%%%%%%%%%%%%%%%%%%%%%%%%%%%%%%%%%%%%%%%%%%%%%%%%
%\appendix
%\section{Some Supporting results}

%%%%%%%%%%%%%%%%%%%%%%%%%%%%%%%%%%%%%%%%%%%%%%%%%%%%%%%%%%%%%%%%%%%%%%%%%%%%%%
%\section{General Discussion Notes}

\bibliographystyle{abbrv}
\bibliography{Caleb.bib,Caleb2.bib,Caleb3.bib,Caleb4.bib,mjh.bib}

%\clearpage
%\input{app}

\vspace*{0.5cm}

\end{document}